\documentclass[12pt]{article}
\usepackage[latin1]{inputenc}
\usepackage[british]{babel}
\usepackage{lmodern}
\usepackage[T1]{fontenc}
\usepackage[paper=a4paper, left=28mm, right=25mm, top=29mm, bottom=25mm]{geometry}
\usepackage{latexsym,amsfonts,amsmath,graphics}
\usepackage{epsfig}
\usepackage{enumitem}
\usepackage{amssymb,mathrsfs}
\usepackage{verbatim}
\newtheorem{theorem}{Theorem}
\newtheorem{lemma}{Lemma}
\newtheorem{corollary}{Corollary}
\newtheorem{example}{Example}
\newtheorem{proposition}{Proposition}

\newtheorem{remark}{Remark}

\newenvironment{proof}{\begin{trivlist} \item[\hskip\labelsep{\it Proof.}]}{$\hfill\Box$\end{trivlist}}

\newcommand{\gr}{\gtrsim}
\newcommand{\lr}{\lesssim}

\newcommand{\rd}{\,\mathrm{d}}
\newcommand{\bsj}{\boldsymbol{j}}
\newcommand{\bsx}{\boldsymbol{x}}

\newcommand{\bsz}{\boldsymbol{z}}

\newcommand{\bsm}{\boldsymbol{m}}
\newcommand{\bst}{\boldsymbol{t}}
\newcommand{\bsa}{\boldsymbol{a}}
\newcommand{\bsc}{\boldsymbol{c}}

\newcommand{\bszero}{\boldsymbol{0}}
\newcommand{\bsone}{\boldsymbol{1}}

\newcommand{\RR}{\mathbb{R}}

\newcommand{\NN}{\mathbb{N}}

\newcommand{\DD}{\mathbb{D}}

\newcommand{\ZZ}{\mathbb{Z}}

\newcommand{\sym}{{\rm sym}}
\newcommand{\cP}{\mathcal{P}}

\newcommand{\vecs}{\boldsymbol{\sigma}}

\allowdisplaybreaks

\title{Finding exact formulas for the $L_2$ discrepancy \\ of digital $(0,n,2)$-nets via Haar functions}
\author{Ralph Kritzinger \thanks{The author is supported by the Austrian Science Fund (FWF): Project F5509-N26, which is a part of the Special Research Program "Quasi-Monte Carlo Methods: Theory and Applications". The support of the Erwin Schr\"{o}dinger International Institute for Mathematics and Physics (ESI) under
the thematic programme "Tractability of High Dimensional Problems and Discrepancy" is gratefully acknowledged.
}}
\date{}

\begin{document}

\maketitle

\begin{abstract}
We use the Haar function system in order to study the $L_2$ discrepancy of a class of digital $(0,n,2)$-nets.
Our approach yields exact formulas for this quantity, which measures the irregularities of distribution of a set of
points in the unit interval. We will obtain such formulas not only for the classical digital nets, but also for shifted and symmetrized versions thereof.
The basic idea of our proofs is to calculate all Haar coefficents of the discrepancy function exactly and insert them
into Parseval's identity. We will also discuss reasons why certain (symmetrized) digital nets fail to achieve the optimal order
of $L_2$ discrepancy and use the Littlewood-Paley inequality in order to obtain results on the $L_p$ discrepancy for all $p\in (1,\infty)$.
\end{abstract} 

\centerline{\begin{minipage}[hc]{130mm}{
}
\end{minipage}}

 \allowdisplaybreaks

\section{Introduction and main results}

In this paper, we study the $L_p$ discrepancy of special digital $(0,n,2)$-nets with a main focus on a precise
computation of the $L_2$ discrepancy. We will introduce all the relevant notation in the following. \\
Discrepancy theory treats the irregularities of point distributions; often in the $d$-dimen-sional unit cube $[0,1)^d$ (see e.g.~\cite{kuinie}).
We study point sets $\cP$ with $N$ elements in the unit square $[0,1)^2$. We define
the discrepancy function of such a point set by
$$ \Delta(\bst,\cP)=\frac{1}{N}\sum_{\bsz\in\cP}\bsone_{[\bszero,\bst)}(\bsz)-t_1t_2, $$
where for $\bst=(t_1,t_2)\in [0,1]^2$ we set $[\bszero,\bst)=[0,t_1)\times [0,t_2)$ with volume $t_1t_2$
and denote by $\bsone_{[\bszero,\bst)}$ the indicator function of this interval. The $L_p$ discrepancy for $p\in [1,\infty)$ of $\cP$ is given by
$$ L_{p}(\cP):=\|\Delta(\cdot,\cP)\|_{L_{p}([0,1)^2)}=\left(\int_{[0,1]^2}|\Delta(\bst,\cP)|^p\rd \bst\right)^{\frac{1}{p}} $$
and the star discrepancy  of $\cP$ is defined as
$$ L_{\infty}(\cP):=\|\Delta(\cdot,\cP)\|_{L_{\infty}([0,1)^2)}=\sup_{\bst \in [0,1]^2}|\Delta(\bst,\cP)|. $$
Throughout this paper, for functions $f,g:\NN \rightarrow \RR^+$, we write $g(N) \lr f(N)$ and $g(N) \gr f(N)$, 
if there exists a $C>0$ such that $g(N) \le C f(N)$ or $g(N) \geq C f(N)$ for all $N \in \NN$, $N \ge 2$, respectively. This constant $C$ is independent of $N$, but might depend on several other parameters $\alpha_1,\dots,\alpha_i$, which we sometimes emphasize by writing $\lr_{\alpha_1,\dots,\alpha_i}$ and $\gr_{\alpha_1,\dots,\alpha_i}$, respectively. Further, we write $f(N) \asymp g(N)$ if the relations $g(N) \lr f(N)$ and $g(N) \gr f(N)$ hold simultaneously. \\
 It is well known that for every $p\in [1,\infty)$ the $L_p$ discrepancy of any point set $\mathcal{P}$ consisting of $N$ points in $[0, 1)^2$ satisfies
\begin{equation} \label{roth} L_p(\mathcal{P}) \gr_p \frac{\sqrt{\log{N}}}{N}, \end{equation}
where $\log$ denotes the natural logarithm.
This was first shown by Roth \cite{Roth2} for $p = 2$ and hence for all $p \in [2,\infty)$ and later by
Schmidt \cite{schX} for all $p\in(1,2)$. The case $p=1$ was added by Hal\'{a}sz \cite{hala}. For the star discrepancy of such a point set $\cP$ we have the best possible lower bound
\begin{equation} \label{schmidt} L_{\infty}(\mathcal{P}) \gr \frac{\log{N}}{N}, \end{equation}
which is due to Schmidt~\cite{Schm72distrib}. \\
An important class of point sets with low star discrepancy are digital nets (see e.g.~\cite{Nied87,DP05}). A digital net in base $2$ is a point set $\{\boldsymbol{x}_0,\ldots, \boldsymbol{x}_{2^n-1}\}$ in the $d$-dimensional unit interval $[0,1)^d$,
which is generated by $d$ matrices of size $n\times n$. Hence we need two matrices to generate a digital net
in the unit square. The procedure is as follows. Let $n\geq 1$ be an integer.
\begin{itemize}
\item Choose a bijection $\varphi:\{0,1\} \rightarrow \mathbb{Z}_2$, where $\ZZ_2$ is the field with two elements.
\item Choose $n \times n$ matrices $C_1$ and $C_2$ over $\ZZ_2$.
\item For some $r\in\{0,1,\dots,2^n-1\}$ let $r=r_0+2r_1  +\cdots +2^{n-1}r_{n-1}$ with $r_i\in\{0,1\}$ for all $i\in\{0,\dots,n-1\}$ be the dyadic expansion of $r$. Map $r$ to the vector $\vec{r}=(\varphi(r_0),\ldots , \varphi(r_{n-1}))^{\top}$.
\item Compute $C_j \vec{r}=:(y_{r,1}^{(j)},\ldots ,y_{r,n}^{(j)})^{\top}$ for $j=1,2$.
\item Compute $x_r^{(j)}=\frac{\varphi^{-1}(y_{r,1}^{(j)})}{2}+\cdots +\frac{\varphi^{-1}(y_{r,n}^{(j)})}{2^n}$ for $j=1,2$.  
\item Set $\boldsymbol{x}_{r}=(x_r^{(1)},x_{r}^{(2)})$.
\item Repeat steps 3 to 6 for all $r\in \{0,1,\dots,2^n-1\}$ and set $\cP:=\{\bsx_1,\dots,\bsx_{2^n-1}\}$. We call $\cP$ a digital net generated by $C_1$ and $C_2$.
\end{itemize}
A point set $\cP$ in the unit square is called a $(0,n,2)$-net in base 2, if every dyadic box 
$$\left[\frac{m_1}{2^{j_1}},\frac{m_1+1}{2^{j_1}}\right) \times \left[\frac{m_2}{2^{j_2}},\frac{m_2+1}{2^{j_2}}\right),$$
where $j_1,j_2\in\NN_0$ and $m_1\in\{0,1,\dots,2^{j_1}-1\}$ and $m_2\in\{0,1,\dots,2^{j_2}-1\}$
 with volume $2^{-n}$, i.e. with $j_1+j_2=n$, contains exactly one element of $\cP$. It is well known that a digital net is a $(0,n,2)$-net if and only if the following condition holds: For every choice of integers $d_1,d_2\in \NN_0$ with $d_1+d_2=n$ the first $d_1$ lines of $C_1$ and the first $d_2$ lines of $C_2$ are linearly independent. By Niederreiter~\cite{Nied87}, we know that the star discrepancy of any $(0,n,2)$-net in base 2 is of best possible order $(\log{N})/N$. In particular, we have by~\cite{Lar} the general upper bound 
 $$ 2^n L_{\infty}(\cP)\leq \frac{n}{3}+\frac{19}{3} $$
 for every digital $(0,n,2)$-net. \\
The situation is less clear for the $L_2$ discrepancy of digital $(0,n,2)$-nets. Classical nets like the Hammersley point set (see Example~\ref{hamm}) fail to achieve the optimal order $\sqrt{\log{N}}/N$ of $L_2$ discrepancy.
To reduce the $L_2$ discrepancy of digital nets, digital shifts have been applied to such nets in many previous papers~\cite{HZ,DP05,Kri2}. A digital shift $\vecs=(\sigma_1,\dots,\sigma_n)^{\top}$ is an
element of $\ZZ_2^n$. We obtain a digit shifted digital net by altering the fourth step in the construction scheme of digital nets to $C_2 \vec{r}+\vecs=:(y_{r,1}^{(2)},\ldots ,y_{r,n}^{(2)})$; hence after multiplication of the matrix $C_2$ with the vector $\vec{r}$ we also add the digital shift, before we transform the vector back to a number in $[0,1)$. Note that by~\cite[Lemma 2.2]{Kri1} we can apply the shift only to the second component without loss of generality.
We consider the following $n\times n$ matrices over $\ZZ_2$:

\begin{equation} \label{matrixa} A_1=
\begin{pmatrix}
0 & 0 & 0 & \cdots & 0 & 0 & 1 \\
0 & 0 & 0 & \cdots & 0 & 1 & 0 \\
0 & 0 & 0 & \cdots & 1 & 0 & 0 \\
\vdots & \vdots & \vdots & \ddots & \vdots & \vdots & \vdots & \\
0 & 0 & 1 & \cdots & 0 & 0 & 0 \\
0 & 1 & 0 & \cdots & 0 & 0 & 0 \\
1 & 0 & 0 & \cdots & 0 & 0 & 0 \\
\end{pmatrix}
\text{\, and\,}
 A_2=
\begin{pmatrix}
1 & 0 & 0 & \cdots & 0 & 0 & a_{1} \\
0 & 1 & 0 & \cdots &  0 & 0 & a_{2} \\
0 & 0 & 1 & \cdots & 0 & 0 & a_{3} \\
\vdots & \vdots & \vdots & \ddots & \vdots & \vdots & \vdots & \\
0 & 0 & 0 & \cdots &  1 & 0 & a_{n-2} \\
0 & 0 & 0 & \cdots &  0 & 1 & a_{n-1} \\
0 & 0 & 0 & \cdots &  0 & 0 & 1 \\
\end{pmatrix}.
\end{equation}
We study the discrepancy of the digital net $\cP_{\bsa}(\vecs)$ with $\bsa=(a_1,\dots,a_{n-1})^T$, generated by $A_1$ and $A_2$ and shifted by $\vecs=(\sigma_1,\dots,\sigma_n)^T$. We simply write $\cP_{\bsa}$ if we do not apply a shift. The set $\cP_{\bsa}(\vecs)$ can be written as
$$ \cP_{\bsa}(\vecs)=\left\{\bigg(\frac{t_n}{2}+\dots+\frac{t_1}{2^n},\frac{b_1}{2}+\dots+\frac{b_n}{2^n}\bigg):t_1,\dots, t_n \in\{0,1\}\right\}, $$
where $b_k=t_k\oplus a_{k}t_{n}\oplus\sigma_n$ for $k\in\{1,\dots,n-1\}$ and $b_n=t_n\oplus \sigma_n$. The operation $\oplus$ denotes addition modulo 2.
 \\
We also consider symmetrized versions of shifted digital nets. It is convenient to define $\widetilde{\cP}_{\bsa}(\vecs)=\cP_{\bsa}(\vecs)\cup \cP_{\bsa}(\vecs^*)$, where
$\vecs^*=(\sigma_1\oplus 1,\dots,\sigma_n\oplus 1)^T$. Note that $\widetilde{\cP}_{\bsa}(\vecs)$ can also be written in the form
$$ \widetilde{\cP}_{\bsa}(\vecs)=\cP_{\bsa}(\vecs)\cup \{(x,1-2^{-n}-y):(x,y)\in \cP_{\bsa}(\vecs)\},$$
which justifies the terminus symmetrized digital net. It has been observed several times before that symmetrization can often reduce the $L_2$ discrepancy
of point sets to the best possible order~\eqref{roth} (see e.g.~\cite{daven,lp,bil}). We will discuss this phenomenon in more detail in Section~\ref{why}. \\
Theorem~\ref{theo1} shows an exact formula for the $L_2$ discrepancy of the class $\cP_{\bsa}(\vecs)$ of shifted digital $(0,n,2)$-nets.

\begin{theorem} \label{theo1}
 Let $L=\sum_{i=1}^{n-1}a_i (1-2\sigma_i)$ and $\ell=\sum_{i=1}^n (1-2\sigma_i)$. Then we have
   \begin{align*}
     (2^n\,L_2(\cP_{\bsa}(\vecs)))^2=& \frac{1}{64}\left((\ell-L)^2+L^2+8\ell-10L+\frac53 n\right) \\
		    &+\frac{1}{2^{n+4}}\left(2\sigma_nL-\ell+4\right)+\frac{3}{8}-\frac19\frac{1}{2^{2n+3}}.
   \end{align*}
	Hence we have $L_2(\cP_{\bsa}(\vecs))\lr \sqrt{\log{N}}/N$ if and only if the conditions $|\ell-L|\lr \sqrt{n}$ and $|L|\lr \sqrt{n}$ hold simultaneously.
\end{theorem}

\begin{remark} \rm
  For a fixed $\bsa\in\ZZ_2^{n-1}$, how can we construct a shift $\vecs\in\ZZ_2^n$ which satisfies $|\ell-L|\lr\sqrt{n}$ and $|L|\lr \sqrt{n}$ simultaneously? Put $I_0:=\{i\in\{1,\dots,n-1\}:a_i=0\}$ and $I_1:=\{i\in\{1,\dots,n-1\}:a_i=1\}$ and further $\ell_0:=|\{i\in I_0: \sigma_i=0\}|$ as well as $\bar{\ell}_0:=|\{i\in I_1: \sigma_i=0\}|$. Choose $\vecs$ such that
	$||I_0|-2\ell_0|\lr\sqrt{n}$ and $||I_1|-2\bar{\ell}_0|\lr \sqrt{n}$; hence the number of zeros and ones in the components of the shifts whose indices belong to $I_0$ or $I_1$ has to be balanced, respectively.
\end{remark}

\begin{example} \rm \label{hamm}
   We study a special instance of our nets, namely the point set $\cP_{\bszero}(\vecs)$, where $\bszero=(0,\dots,0)\in\ZZ_2^{n-1}$. This point set
   is the (digit shifted) Hammersley point set in base 2 (which is also known as van der Corput set or Roth net). For $\bsa=\bszero$ we have $L=0$ and 
   $\ell=\sum_{i=1}^n (1-2\sigma_i)=\sum_{i=1}^n (2(1-\sigma_i)-1)=2z-n$, where $z$ denotes the number of zero digits in the digital shift $\vecs$.
   We insert these results into Theorem~\ref{theo1} and find
   $$ (L_2(\cP_{\bszero}(\vecs)))^2=\frac{n^2}{64}+\frac{z^2}{16}-\frac{zn}{16}-\frac{19n}{192}+\frac{z}{4}+\frac{n}{2^{n+4}}-\frac{z}{2^{n+3}}+\frac{1}{2^{n+2}}+\frac38-\frac{1}{9\cdot 2^{2n+2}}. $$
   This formula has already been obtained by Kritzer and Pillichshammer~\cite[Theorem 1]{Kri2} in 2006. Their proof is different from ours, since they used an explicit formula for the discrepancy function of the digit shifted Hammersley point set, which has been found by Larcher and Pillichshammer~\cite[Example 2]{Lar} in 2001 by an approach via Walsh functions. Like Haar functions, which will be the central tool used in this paper, Walsh functions are also an orthonormal basis of the $L_2([0,1)^2)$ space and are useful in order to study the $L_2$ discrepancy of digital nets. For more details on Walsh functions we refer to~\cite[Appendix A]{DP10}.
\end{example}  
As an immediate corollary of Theorem~\ref{theo1} we compute the $L_2$ discrepancy of the unshifted nets. We observe the surprising fact that the $L_2$ discrepancy only
depends on the number of zeros and ones in $\bsa$, but not on their positions. The result follows from Theorem~\ref{theo1} by setting $\sigma_i=0$ for all
$i=1,\dots,n$, which yields $L=\sum_{i=1}^{n-1}a_i$ and $\ell=n$.
\begin{corollary} \label{coro1}
   Let $|\bsa|=\sum_{i=1}^{n-1}a_i$. Then we have
    \begin{align*}
     (2^n\,L_2(\cP_{\bsa}))^2=& \frac{1}{64}\left((n-|\bsa|)^2+|\bsa|^2-10|\bsa|+\frac{29}{3} n\right)+\frac{3}{8}-\frac{n-4}{2^{n+4}}-\frac19\frac{1}{2^{2n+3}}.
   \end{align*}
	Hence we have $L_2(\cP_{\bsa})\gr (\log{N})/N$ for all $\bsa\in\ZZ_2^{n-1}$.
\end{corollary}
Now we fix $\bsa$ and ask how large the $L_2$ discrepancy of the shifted nets $\cP_{\bsa}(\vecs)$ is in average. In other words, we compute the mean of $(2^nL_2(\cP_{\bsa}(\vecs)))^2$ over all possible shifts $\vecs\in \ZZ_2^n$.

\begin{corollary} \label{average} 
  Let $\bsa\in\ZZ_2^{n-1}$ be fixed. Then we have
    \begin{align*}
     \frac{1}{2^n}\sum_{\vecs\in\ZZ_2^n}(2^n\,L_2(\cP_{\bsa}(\vecs)))^2=&\frac{n}{24}+\frac38+\frac{1}{2^{n+2}}-\frac{1}{9\cdot 2^{2n+3}}.
   \end{align*}
	Hence the mean of the squared $L_2$ discrepancy of $\cP_{\bsa}(\vecs)$ over all possible shifts $\vecs\in \ZZ_2^n$ is the same for all $\bsa\in\ZZ_2^{n-1}$ and of best possible order according to~\eqref{roth}.
\end{corollary}

\begin{proof}
 It is not difficult to verify $\frac{1}{2^n}\sum_{\vecs\in\ZZ_2^n}(\ell-L)^2=n-|\bsa|$ and $\frac{1}{2^n}\sum_{\vecs\in\ZZ_2^n}L^2=|\bsa|$ as well as $\frac{1}{2^n}\sum_{\vecs\in\ZZ_2^n}\ell=\frac{1}{2^n}\sum_{\vecs\in\ZZ_2^n}L=0$, which leads to the result.
\end{proof}

\begin{remark} \rm
    Dick and Pillichshammer studied the problem of the mean squared $L_2$ discrepancy of digital nets in~\cite{DP05}. They did not only apply a shift $\vecs\in\ZZ_2^n$ to the first $n$ digits of the coordinates as introduced in this paper, but also added random numbers from $[0,2^{-n})$ to each component of all elements of the digital net after the shifting process. Then they computed the mean over all shifts and obtained the same result for every digital $(0,n,2)$-net. The authors also studied the problem in higher dimensions. With the methods used in~\cite{DP05} one can show that Corollary~\ref{average} actually holds for all digital $(0,n,2)$-nets.
\end{remark}
We will prove the following exact result concerning the $L_2$ discrepancy of the symmetrized nets $\widetilde{\cP}_{\bsa}(\vecs)$. This formula demonstrates that the $L_2$ discrepancy depends on $\bsa$ and on $\vecs$, but only to a minor extent.

\begin{theorem} \label{theo2}
   Let $\widetilde{\cP}_{\bsa}(\vecs)$ with $2^{n+1}$ elements. Then we have
      $$ (2^{n+1}L_2(\widetilde{\cP}_{\bsa}(\vecs)))^2=\frac{n}{24}+\frac{11}{8}+\frac{1}{2^n}-\frac{1}{9\cdot 2^{2n+1}}-\frac{(-1)^{\sigma_n}}{2^{n+2}}L. $$
	Hence the point sets $\widetilde{\cP}_{\bsa}(\vecs)$ achieve the optimal order of $L_2$ discrepancy without any conditions on $\bsa$ and $\vecs$.
\end{theorem}

\begin{remark} \rm
Again, the $L_2$ discrepancy of the unshifted symmetrized nets depend only on the parameter $|\bsa|$, as we derive 
$$ (2^{n+1}L_2(\widetilde{\cP}_{\bsa}))^2=\frac{n}{24}+\frac{11}{8}+\frac{1}{2^n}-\frac{1}{9\cdot 2^{2n+1}}-\frac{1}{2^{n+2}}|\bsa|. $$
For the symmetrized shifted Hammersley point set $\widetilde{\cP}_{\bszero}(\vecs)$ we obtain
$$ (2^{n+1}L_2(\widetilde{\cP}_{\bszero}(\vecs)))^2=\frac{n}{24}+\frac{11}{8}+\frac{1}{2^n}-\frac{1}{9\cdot 2^{2n+1}} $$
and the fact that the $L_2$ discrepancy is independent of the shift $\vecs$. This result has previously been obtained by the author in~\cite{Kritz} with the methods used in~\cite{Lar} and~\cite{Kri2}. Further we immediately obtain for every $\bsa\in\ZZ_2^{n-1}$ the average result
$$\frac{1}{2^n}\sum_{\vecs\in\ZZ_2^n}(2^n\,L_2(\widetilde{\cP}_{\bsa}(\vecs)))^2=\frac{n}{24}+\frac{11}{8}+\frac{1}{2^n}-\frac{1}{9\cdot 2^{2n+1}}.$$
Note that the fact that the nets $\widetilde{\cP}_{\bsa}$ achieve the optimal order of $L_2$ discrepancy independently of $\bsa$ follows already from~\cite[Theorem 2]{Kritz}.
\end{remark}

\begin{remark}  \rm
  Since the proofs of Theorem~\ref{theo1} and~\ref{theo2} as presented in Sections 3 and 4 are very technical and prone to mistakes, we tested the correctness of our formulas with Warnock's formula~\cite{Warn}.	It states that for a point set $\cP=\{\bsx_0,\dots,\bsx_{N-1}\}$ in the unit square with
$\bsx_k=(x_{k,1},x_{k,2})$ for $k=0,\dots,N-1$ we have
\begin{align*} 
  (N\,L_{2,N}(\cP))^2
  = \frac{N^2}{9}-\frac{N}{2}\sum_{k=0}^{N-1}\prod_{i=1}^2 (1-x_{k,i}^2)+\sum_{k,l=0}^{N-1}\prod_{i=1}^2 \left(1-\max\{x_{k,i},x_{l,i}\}\right).
\end{align*}
This formula allows us to compute the $L_2$ discrepancy of $\cP_{\bsa}(\vecs)$ exactly, provided that the number of points $N=2^n$ is small (e.g. $n=10$).
Then we can compare the results of Warnock's formula with the output of our formulas and we always observe a match. Note that Warnock's formula requires $\mathcal{O}(N^2)$
operations to compute the $L_2$ discrepancy of a given point set, whereas our formulas allow a very fast computation of this quantity for $\cP_{\bsa}(\vecs)$ and $\widetilde{\cP}_{\bsa}(\vecs)$. 
\end{remark}

We would like to close the introduction by pointing out three papers which heavily influenced the current paper. The first one is~\cite{Kri2} by Kritzer and Pillichshammer,
who obtained the exact result for the $L_2$ discrepancy of the shifted Hammersley point set and discovered the beautiful fact that it only depends on the number
of zeroes in the shift $\vecs$ but not on their position. It is a natural question whether this result can also be obtained with reasonable effort by using Haar functions, as Hinrichs~\cite{hin2010} computed the Haar coefficients of the corresponding discrepancy function exactly in almost all cases. However, the aim of his paper was to estimate the Besov norm of the discrepancy function, and therefore he was content with upper bounds rather than exact formulas in certain cases. We apply the notation of~\cite{hin2010} in this paper and use several results and ideas from there. The third paper which inspired this work is by Bilyk, Temlyakov and Yu~\cite{bil}, who computed the Fourier coefficients of the discrepancy function of the symmetrized Fibonacci lattice exactly in order to find an exact formula for its $L_2$ discrepancy. We do the same for a class of digital $(0,n,2)$-net with the difference that we compute the Haar coefficients instead of the Fourier coefficients, since Haar functions fit the structure of digital nets much better than harmonic functions. \\
The outline of this paper is as follows. In Section~\ref{haarf} we introduce the Haar function system and present general formulas for the Haar coefficients of the discrepancy function of arbitrary point sets in the unit square. Section~\ref{haark} is the longest and most technical section, in which we will compute all the Haar coefficients of $\Delta(\cdot,\cP_{\bsa}(\vecs))$ exactly and insert them into Parseval's identity in order to prove Theorem~\ref{theo1}. In Section 4 we do the same for the discrepancy function of the symmetrized nets, but we omit all the technical details. In Section 5 we will comment on the results for the Haar coefficients in the previous sections. In particular, we point out which Haar coefficients cause a large $L_2$ discrepancy of (symmetrized) digital nets. We will disprove a conjecture by Bilyk and give a new proof of a result by Larcher and Pillichshammer on symmetrized nets. In Section 6 we consider a different class of digital nets, for the $L_2$ discrepancy of which we can also find an exact formula with the same method as in Section~\ref{haark}. We therefore omit technicalities again in this section. In Section 7 we discuss the $L_p$ discrepancy of digital nets with the aid of a Littlewood-Paley inequality and in the final Section 8 we mention several open problems which we would like to investigate in future research.

\section{The Haar expansion of the discrepancy function} \label{haarf}

A dyadic interval of length $2^{-j}, j\in {\mathbb N}_0,$ in $[0,1)$ is an interval of the form 
$$ I=I_{j,m}:=\left[\frac{m}{2^j},\frac{m+1}{2^j}\right) \ \ \mbox{for } \  m=0,1,\ldots,2^j-1.$$ 
The left and right half of $I_{j,m}$ are the dyadic intervals $I_{j+1,2m}$ and $I_{j+1,2m+1}$, respectively. The Haar function $h_{j,m}$  
is the function on $[0,1)$ which is  $+1$ on the left half of $I_{j,m}$, $-1$ on the right half of $I_{j,m}$ and 0 outside of $I_{j,m}$. The $L_\infty$-normalized Haar system consists of
all Haar functions $h_{j,m}$ with $j\in{\mathbb N}_0$ and  $m=0,1,\ldots,2^j-1$ together with the indicator function $h_{-1,0}$ of $[0,1)$.
Normalized in $L_2([0,1))$ we obtain the orthonormal Haar basis of $L_2([0,1))$. 

Let ${\mathbb N}_{-1}=\NN_0 \cup \{-1\}$ and define ${\mathbb D}_j=\{0,1,\ldots,2^j-1\}$ for $j\in{\mathbb N}_0$ and ${\mathbb D}_{-1}=\{0\}$.
For $\bsj=(j_1,j_2)\in{\mathbb N}_{-1}^2$ and $\bsm=(m_1,m_2)\in {\mathbb D}_{\bsj} :={\mathbb D}_{j_1} \times {\mathbb D}_{j_2}$, 
the Haar function $h_{\bsj,\bsm}$ is given as the tensor product 
$$h_{\bsj,\bsm}(\bst) = h_{j_1,m_1}(t_1) h_{j_2,m_2}(t_2) \ \ \ \mbox{ for } \bst=(t_1,t_2)\in[0,1)^2.$$
We speak of $I_{\bsj,\bsm} = I_{j_1,m_1} \times I_{j_2,m_2}$ as dyadic boxes with level $|\bsj|=\max\{0,j_1\}+\max\{0,j_2\}$, where we set $I_{-1,0}=\bsone_{[0,1)}$. The system
$$ \{2^{\frac{|\bsj|}{2}}h_{\bsj,\bsm}: \bsj\in\NN_{-1}^2, \bsm\in \DD_{\bsj}\} $$
is an orthonormal basis of $L_2([0,1)^2)$ and we have Parseval's identity which states that for every function $f\in L_2([0,1)^2)$ we have
\begin{equation} \label{parseval}
   \|f\|_{L_2([0,1)^2)}^2=\sum_{\bsj\in \NN_{-1}^2} 2^{|\bsj|} \sum_{\bsm\in\DD_{\bsj}} |\mu_{\bsj,\bsm}|^2,
\end{equation}
where the numbers $\mu_{\bsj,\bsm}=\mu_{\bsj,\bsm}(f)=\langle f, h_{\bsj,\bsm} \rangle =\int_{[0,1)^2} f(\bst) h_{\bsj,\bsm}(\bst)\rd\bst$ are the so-called Haar coefficients of $f$. \\
Let $\cP$ be an arbitrary $2^n$-element point set in the unit square. The Haar coefficients of its discrepancy function $\Delta(\cdot,\cP)$ are as follows (see~\cite{hin2010}). By $\bsz\in I_{\bsj,\bsm}$ we actually mean $\bsz=(z_1,z_2)\in I_{\bsj,\bsm} \cap \cP.$

\begin{itemize}
  \item If $\bsj=(-1,-1)$, then 
   \begin{equation} \label{art1} \mu_{\bsj,\bsm}=2^{-n}\sum_{\bsz\in \cP} (1-z_1)(1-z_2)-\frac14. \end{equation}
   \item If $\bsj=(j_1,-1)$ with $j_1\in \NN_0$, then 
   \begin{equation} \label{art2} \mu_{\bsj,\bsm}=-2^{-n-j_1-1}\sum_{\bsz\in I_{\bsj,\bsm}} (1-|2m_1+1-2^{j_1+1}z_1|)(1-z_2)+2^{-2j_1-3}. \end{equation}
    \item If $\bsj=(-1,j_2)$ with $j_2\in \NN_0$, then 
   \begin{equation} \label{art3} \mu_{\bsj,\bsm}=-2^{-n-j_2-1}\sum_{\bsz\in I_{\bsj,\bsm}} (1-|2m_2+1-2^{j_2+1}z_2|)(1-z_1)+2^{-2j_2-3}. \end{equation}
    \item If $\bsj=(j_1,j_2)$ with $j_1,j_2\in \NN_0$, then 
   \begin{align} \label{art4} \mu_{\bsj,\bsm}=&2^{-n-j_2-j_2-2}\sum_{\bsz\in I_{\bsj,\bsm}} (1-|2m_1+1-2^{j_1+1}z_1|)(1-|2m_2+1-2^{j_2+1}z_2|) \nonumber \\ &-2^{-2j_1-2j_2-4}. \end{align}
   \end{itemize}
Note that we could also write $\bsz\in \mathring{I}_{\bsj,\bsm}$, where $\mathring{I}_{\bsj,\bsm}$ denotes the interior of $I_{\bsj,\bsm}$, since the summands in the formulas~\eqref{art2}--\eqref{art4} vanish if $\bsz$ lies on the boundary of the dyadic box. Hence, in order to compute the Haar coefficients of the discrepancy function, we have to deal with the sums over $\bsz$ which appear in the formulas above and to determine which points $\bsz=(z_1,z_2)\in \cP$ lie in the dyadic box $I_{\bsj,\bsm}$ with $\bsj\in \NN_{-1}^2$ and $\bsm=(m_1,m_2)\in\DD_{\bsj}$. If $m_1$ and $m_2$ are nonnegative integers, then they have a dyadic expansion of the form
\begin{equation} \label{mdyadic} m_1=2^{j_1-1}r_1+\dots+r_{j_1}  \text{\,  and  \,}  m_2=2^{j_2-1}s_1+\dots+s_{j_2}  \end{equation}
with digits $r_{i_1},s_{i_2}\in\{0,1\}$ for all $i_1\in\{1,\dots,j_1\}$ and $i_2\in\{1,\dots,j_2\}$, respectively.
Let $\bsz=(z_1,z_2)=\big(\frac{t_n}{2}+\dots+\frac{t_1}{2^n},\frac{b_1}{2}+\dots+\frac{b_n}{2^n}\big)$ be a point of our point set $\cP_{\bsa}(\vecs)$. Then $\bsz\in I_{\bsj,\bsm}$
if and only if 
\begin{equation} \label{cond} t_{n+1-k}=r_k \text{\, for all \,} k\in \{1,\dots, j_1\} \text{\, and \,} b_k=s_k \text{\, for all \,} k\in \{1,\dots, j_2\}. \end{equation}
Further, for such a point $\bsz=(z_1,z_2)\in I_{\bsj,\bsm}$ we have
\begin{equation} \label{z1} 2m_1+1-2^{j_1+1}z_1=1-t_{n-j_1}-2^{-1}t_{n-j_1-1}-\dots-2^{j_1-n+1}t_1 \end{equation}
and
\begin{equation} \label{z2} 2m_2+1-2^{j_2+1}z_2=1-b_{j_2+1}-2^{-1}b_{j_2+2}-\dots-2^{j_2-n+1}b_n. \end{equation}
These observations will be the starting point of all proofs in the following section.

\section{The Haar coefficients of the discrepancy function of $\cP_{\bsa}(\vecs)$} \label{haark}

Recall the definitions of $\ell$ and $L$ from Theorem~\ref{theo1}. Throughout the whole section, by $\sigma_j'$ for $j\in\{1,\dots,n-1\}$ we shall always mean $\sigma_j\oplus a_j$.
The idea for the proof of Theorem~\ref{theo1} is as follows: We partition the set $\NN_{-1}^2$ in 13 smaller sets $\mathcal{J}_i$ for $i=1,\dots,13$. Then we compute the Haar coefficients $\mu_{\bsj,\bsm}$ of $\Delta(\cdot,\cP_{\bsa}(\vecs))$ for
all $\bsj\in\mathcal{J}_i$ and further $\sum_{\bsj\in\mathcal{J}_i}2^{|\bsj|}\sum_{\bsm\in\DD_{\bsj}}|\mu_{\bsj,\bsm}|^2$. Then Theorem~\ref{theo1} follows via Parseval by
$$ (2^n\,L_2(\cP_{\bsa}(\vecs)))^2=\sum_{i=1}^{13} \sum_{\bsj\in\mathcal{J}_i}2^{|\bsj|}\sum_{\bsm\in\DD_{\bsj}}|\mu_{\bsj\bsm}|^2. $$

\paragraph{Case 1: $\bsj\in\mathcal{J}_1:=\{(-1,-1)\}$}

\begin{proposition} \label{prop1}
  Let $\bsj\in \mathcal{J}_1$ and $\bsm\in \DD_{\bsj}$. Then we have
     $$ \mu_{\bsj,\bsm}=\frac{1}{2^{n+1}}+\frac{1}{2^{2n+2}}+\frac{1}{2^{n+3}}(\ell-L). $$
\end{proposition}

\begin{proof}
  By~\eqref{art1} we have
	\begin{align*}
	   \mu_{\bsj,\bsm}=& 2^{-n}\sum_{\bsz\in \cP} (1-z_1)(1-z_2)-\frac14 \\
		  =&1-2^{-n}\sum_{\bsz\in I_{\bsj,\bsm}}z_1-2^{-n}\sum_{\bsz\in I_{\bsj,\bsm}}z_2+2^{-n}\sum_{\bsz\in I_{\bsj,\bsm}}z_1z_2-\frac14 \\
			=&-\frac14+2^{-n}+2^{-n}\sum_{\bsz\in I_{\bsj,\bsm}}z_1z_2,
	\end{align*}
	where we regarded $\sum_{\bsz\in I_{\bsj,\bsm}}z_1=\sum_{\bsz\in I_{\bsj,\bsm}}z_2=\sum_{l=0}^{2^n-1}l/2^n=2^{n-1}-2^{-1}$ in the last step. We write $u=2^{-1}t_{n-1}+\dots+2^{-n+1}t_1$ and $v_1=2^{-1}(t_1\oplus \sigma_1)+\dots+2^{n-1}(t_{n-1}\oplus \sigma_{n-1})$ as well as $v_2=2^{-1}(t_1\oplus \sigma_1')+\dots+2^{n-1}(t_{n-1}\oplus \sigma_{n-1}')$ and consider 
	\begin{align*}
	  \sum_{\bsz\in I_{\bsj,\bsm}}z_1z_2=& \sum_{t_1,\dots,t_n=0}^1 \left(\frac{t_n}{2}+\dots+\frac{t_{1}}{2^n}\right)\left(\frac{t_1\oplus a_1t_n\oplus \sigma_1}{2}+\dots+\frac{t_n\oplus\sigma_n}{2^n}\right) \\
		=& \sum_{t_1,\dots,t_{n-1}=0}^1 \Bigg(\frac{u}{2} \bigg(v_1+\frac{\sigma_n}{2^n}\bigg)+\bigg(\frac{1}{2}+\frac{u}{2}\bigg)\bigg(v_2+\frac{\sigma_n\oplus 1}{2^n}\bigg)\Bigg) \\
		=& \sum_{t_1,\dots,t_{n-1}=0}^1 \left(2^{-n-1}-2^{-n-1}\sigma_n+2^{-n-1}u+\frac{v_2}{2}+\frac12 (uv_1+uv_2)\right) \\
		=& 2^{n-1}(2^{-n-1}-2^{-n-1}\sigma_n)+(2^{n-2}-2^{-1})(2^{-n-1}+2^{-1})\\&+\frac12 \sum_{t_1,\dots,t_{n-1}=0}^1(uv_1+uv_2),
	\end{align*}
	where we use the fact that $\sum_{t_1,\dots,t_{n-1}=0}^1 u=\sum_{t_1,\dots,t_{n-1}=0}^1v_2=\sum_{l=0}^{2^{n-1}-1}l/2^{n-1}=2^{n-2}-2^{-1}$ in the last step. We have
	\begin{align}
	   \sum_{t_1,\dots,t_{n-1}=0}^1 uv_1 =&\sum_{t_1,\dots,t_{n-1}=0}^1 \left(\frac{t_{n-1}}{2}+\dots+\frac{t_{1}}{2^{n-1}}\right)\left(\frac{t_1\oplus \sigma_1}{2}+\dots+\frac{t_{n-1}\oplus\sigma_{n-1}}{2^n}\right)  \nonumber \\
		 =& \sum_{t_1,\dots,t_{n-1}=0}^1 \left(\sum_{k=1}^{n-1} \frac{t_k(t_k\oplus \sigma_k)}{2^{n-k}2^k}+\sum_{\substack{k_1, k_2=1 \\ k_1\neq k_2}}^{n-1}\frac{t_{k_1}(t_{k_2}\oplus \sigma_{k_2})}{2^{n-k_1}2^{k_2}}\right)  \nonumber\\
		 =& \frac{1}{2^{n}}\sum_{k=1}^{n-1} 2^{n-2}\sum_{t_k=0}^{1}t_k(t_k\oplus \sigma_k)+\frac{1}{2^{n}}\sum_{\substack{k_1, k_2=1 \\ k_1\neq k_2}}^{n-1}2^{k_1-k_2}2^{n-3}\sum_{t_{k_1},t_{k_2}=0}^{1}t_{k_1}(t_{k_2}\oplus \sigma_{k_2}) \nonumber\\
		=&\frac14 \sum_{k=1}^{n-1} (1\oplus \sigma_k)+\frac18 \sum_{\substack{k_1, k_2=1 \\ k_1\neq k_2}}^{n-1}2^{k_1-k_2}. \label{formelv1}
	\end{align}
	Analogously, we find
	\begin{equation} \label{formelv2} \sum_{t_1,\dots,t_{n-1}=0}^1 uv_2=\frac14 \sum_{k=1}^{n-1} (1\oplus\sigma_k')+\frac18 \sum_{\substack{k_1, k_2=1 \\ k_1\neq k_2}}^{n-1}2^{k_1-k_2} \end{equation}
	and therefore
	$$  \sum_{t_1,\dots,t_{n-1}=0}^1 (uv_1+uv_2)=\frac14 \sum_{k=1}^{n-1} \left(1\oplus \sigma_k+1\oplus \sigma_k'\right)+\frac14 \sum_{\substack{k_1, k_2=1 \\ k_1\neq k_2}}^{n-1}2^{k_1-k_2}. $$
	If $a_k=0$, then $1\oplus \sigma_k+1\oplus \sigma_k'=2-2\sigma_k$ and if $a_k=1$ then $1\oplus \sigma_k+1\oplus \sigma_k'=1$; hence $1\oplus \sigma_k+1\oplus \sigma_k'=(1-a_k)(1-2\sigma_k)+1$ and
	$$ \sum_{k=1}^{n-1} \left(1\oplus \sigma_k+1\oplus \sigma_k'\right)=\ell-(1-2\sigma_n)-L+n-1. $$
	Further, a direct calculation yields
	$$ \sum_{\substack{k_1, k_2=1 \\ k_1\neq k_2}}^{n-1}2^{k_1-k_2} = \sum_{k_1, k_2=1}^{n-1}2^{k_1-k_2}-\sum_{k=1}^{n-1}1=2^n-n-3+2^{-n+2}. $$
	Now we put everything together to arrive at the desired result.
\end{proof}
The following consequence is immediate.
\begin{lemma}
 We have 
 $$\sum_{\bsj\in\mathcal{J}_1}2^{|\bsj|}\sum_{\bsm\in\DD_{\bsj}}|\mu_{\bsj,\bsm}|^2=\left(\frac{1}{2^{n+1}}+\frac{1}{2^{2n+2}}+\frac{1}{2^{n+3}}(\ell-L)\right)^2.$$
\end{lemma}

\paragraph{Case 2: $\bsj\in\mathcal{J}_2:=\{(-1,j_2): 0\leq j_2 \leq n-2\}$}

\begin{proposition}
  Let $\bsj\in \mathcal{J}_2$ and $\bsm\in \DD_{\bsj}$. Then we have
     $$ \mu_{\bsj,\bsm}=2^{-2n-2}-2^{-n-j_2-3}-2^{-2n-1}(\sigma_{j_2+1}\oplus a_{j_1+1}\sigma_n)+2^{-2j_2-3}\sum_{k=1}^{j_2}\frac{s_k\oplus \sigma_k+s_k\oplus \sigma_k'}{2^{n+1-k}}. $$
\end{proposition}

\begin{proof}
   For $\bsz\in I_{\bsj,\bsm}$ we have $b_k=s_k \text{\, for all \,} k\in \{1,\dots, j_2\}$ and therefore
	 \begin{align*}1-z_1=&1-\frac{t_n}{2}-\dots-\frac{t_1}{2^n} \\
	   =&1-\frac{t_n}{2}-\dots-\frac{t_{j_2+1}}{2^{n-j_2}}-\frac{s_{j_2}\oplus a_{j_2}t_n \oplus  \sigma_{j_2}}{2^{n-j_2+1}}-\dots-\frac{s_1\oplus a_1t_n \oplus \sigma_1}{2^n} \\
		 =& 1-u-\frac{t_{j_2+1}}{2^{n-j_2}}-\varepsilon(m_2,t_n),
	\end{align*}
	where $u:=\frac{t_n}{2}-\dots-\frac{t_{j_2+2}}{2^{n-j_2-1}}$ and $\varepsilon:=\frac{s_{j_2}\oplus a_{j_2}t_n \oplus  \sigma_{j_2}}{2^{n-j_2+1}}+\dots+\frac{s_1\oplus a_1t_n \oplus \sigma_1}{2^n}$.
	Further, we have
	\begin{align*}
	 1-|2m_2+1-2^{j_2+1}z_2|=& 1-|1-b_{j_2+1}-\dots-2^{j_2-n+1}b_n|\\ =&\begin{cases}
	                                                                    v & \text{if \,} b_{j_2+1}=0; \text{\, i.e. \,} t_{j_2+1}=a_{j_2+1}t_n\oplus \sigma_{j_2+1}, \\
																																			1-v & \text{if \,} b_{j_2+1}=1; \text{\, i.e. \,} t_{j_2+1}=a_{j_2+1}t_n\oplus \sigma_{j_2+1}\oplus 1,
	                                                                \end{cases}
		\end{align*}
		where $v=v(t_n)=2^{-1}b_{j_2+2}+\dots+2^{j_2-n+1}b_n$. We fix the digits $t_{j_2+2},\dots,t_n$; hence $\varepsilon(m_2,t_n)$ is fixed by $m_2$ and $u$ and $v$ are fixed as well.
		Then we have
		\begin{align*}
		  \sum_{t_{j_2+1}=0}^{1}&(1-z_1)(1-|2m_2+1-2^{j_2+1}z_2|)\\ =&\left(1-u-\frac{a_{j_2+1}t_n\oplus \sigma_{j_2+1}}{2^{n-j_2}}-\varepsilon(m_2,t_n)\right)v \\
			 &+\left(1-u-\frac{a_{j_2+1}t_n\oplus \sigma_{j_2+1}\oplus 1}{2^{n-j_2}}-\varepsilon(m_2,t_n)\right)(1-v) \\
			=& 1-2^{-n+j_2}-\varepsilon(m_1,t_n)-u+2^{-n+j_2}v-2^{-n+j_2}(a_{j_2+1}t_n \oplus \sigma_{j_2+1}) (2v-1).
		\end{align*}
		We sum the last expression over the remaining digits $t_{j_2+2},\dots,t_n$ and regard the fact that 
		$$\sum_{t_{j_2+2},\dots,t_n=0}^{1}v=\sum_{t_{j_2+2},\dots,t_n=0}^{1}u=\sum_{l=0}^{2^{n-j_2-1}-1}\frac{l}{2^{n-j_2-1}}=2^{n-j_2-2}-2^{-1}.$$
		Hence we obtain
		\begin{align*} \sum_{\bsz\in I_{\bsj,\bsm}}&(1-z_1)(1-|2m_2+1-2^{j_2+1}z_2|) \\ =&\frac14(2^{n-j_2}-2^{-n+j_2+1}+1)-2^{-n+j_2}\sum_{t_{j_2+2},\dots,t_n=0}^{1}(a_{j_2+1}t_n \oplus \sigma_{j_2+1}) (2v-1)\\ &-\sum_{t_{j_2+2},\dots,t_n=0}^{1}\varepsilon(m_1,t_n). \end{align*}
  From the definition of $\varepsilon(m_2,t_n)$ it is easy to see that 
	$$\sum_{t_{j_2+2},\dots,t_n=0}^{1}\varepsilon(m_1,t_n)=2^{n-j_2-2}\sum_{k=1}^{j_2}\frac{s_k \oplus \sigma_k+s_k \oplus \sigma_k'}{2^{n+1-k}}.$$
	We compute $\sum_{t_{j_2+2},\dots,t_n=0}^{1}a_{j_2+1}t_n \oplus \sigma_{j_2+1} (2v-1)$ and distinguish the cases $a_{j_2+1}=0$ and $a_{j_2+1}=1$. If $a_{j_2+1}=0$, we obtain
	\begin{align*}
	   \sum_{t_{j_2+2},\dots,t_n=0}^{1}&(a_{j_2+1}t_n \oplus \sigma_{j_2+1}) (2v-1) \\
		    =& \sum_{t_{j_2+2},\dots,t_n=0}^{1}\sigma_{j_2+1} \left(2\left(\sum_{k=j_2+2}^{n-1}\frac{t_k\oplus a_kt_n \oplus \sigma_k}{2^{k-j_2-1}}+\frac{t_n \oplus \sigma_n}{2^{n-j_2-1}}\right)-1\right) \\
				=& \sum_{t_{j_2+2},\dots,t_{n-1}=0}^{1}\sigma_{j_2+1}\Bigg\{ \left(2\left(\sum_{k=j_2+2}^{n-1}\frac{t_k \oplus \sigma_k}{2^{k-j_2-1}}+\frac{\sigma_n}{2^{n-j_2-1}}\right)-1\right) \\
				&+\left(2\left(\sum_{k=j_2+2}^{n-1}\frac{t_k\oplus \sigma_k'}{2^{k-j_2-1}}+\frac{1 \oplus \sigma_n}{2^{n-j_2-1}}\right)-1\right)\Bigg\} \\
				=&\sigma_{j_2+1}\sum_{l=0}^{2^{n-j_2-2}-1}\left\{2\left(\frac{l}{2^{n-j_2-2}}+\frac{\sigma_n}{2^{n-j_2-1}}\right)-1+2\left(\frac{l}{2^{n-j_2-2}}+\frac{1-\sigma_n}{2^{n-j_2-1}}\right)-1\right\} \\
				=& -\sigma_{j_2+1}=-\sigma_{j_2+1}\oplus a_{j_2+1}\sigma_n.
	\end{align*}
	If $a_{j_2+1}=1$, then we get
	\begin{align*}
	   \sum_{t_{j_2+2},\dots,t_n=0}^{1}&(a_{j_2+1}t_n \oplus \sigma_{j_2+1}) (2v-1) \\
		    =& \sum_{t_{j_2+2},\dots,t_n=0}^{1}(t_n\oplus \sigma_{j_2+1}) \left(2\left(\sum_{k=j_2+2}^{n-1}\frac{t_k\oplus a_kt_n \oplus \sigma_k}{2^{k-j_2-1}}+\frac{t_n \oplus \sigma_n}{2^{n-j_2-1}}\right)-1\right) \\
				=& \sum_{t_{j_2+2},\dots,t_n=0}^{1} \left(2\left(\sum_{k=j_2+2}^{n-1}\frac{t_k\oplus a_k(\sigma_{j_2+1}\oplus 1) \oplus \sigma_k}{2^{k-j_2-1}}+\frac{\sigma_{j_2+1}\oplus 1 \oplus \sigma_n}{2^{n-j_2-1}}\right)-1\right) \\
			 =& \sum_{l=0}^{2^{n-j_2-2}-1}\left(2\left(\frac{l}{2^{n-j_2-2}}+\frac{\sigma_{j_1+1}\oplus\sigma_n\oplus 1}{2^{n-j_2-2}}\right)-1\right) \\
				=& \sigma_{j_2+1}\oplus \sigma_n\oplus 1 -1=-\sigma_{j_2+1}\oplus \sigma_n=-\sigma_{j_2+1}\oplus a_{j_2+1}\sigma_n.
	\end{align*}
	Thus, in any case we have $\sum_{t_{j_2+2},\dots,t_n=0}^{1}(a_{j_2+1}t_n \oplus \sigma_{j_2+1}) (2v-1)=-\sigma_{j_2+1}\oplus a_{j_2+1}\sigma_n$ and we arrive at
	\begin{align*} &\sum_{\bsz\in I_{\bsj,\bsm}}(1-z_1)(1-|2m_2+1-2^{j_2+1}z_2|) \\ =&\frac14(2^{n-j_2}-2^{-n+j_2+1}+1)+2^{-n+j_2}(\sigma_{j_2+1}\oplus a_{j_2+1}\sigma_n)-2^{n-j_2-2}\sum_{k=1}^{j_2}\frac{s_k \oplus \sigma_k+s_k \oplus \sigma_k'}{2^{n+1-k}}. \end{align*}
	The rest follows with~\eqref{art3}. 
	
\end{proof}

\begin{lemma}
 We have 
 \begin{align*}\sum_{\bsj\in \mathcal{J}_2}2^{|\bsj|}\sum_{\bsm\in\DD_{\bsj}}|\mu_{\bsj,\bsm}|^2=&\frac19 2^{-4n-6}\left(2n2^{2n}-9(n-1)2^{n+2}+2^{2n+3}-44\right) \\ &+2^{-3n-3}\left(\sum_{i=1}^{n-1}\sigma_i+\sigma_nL\right)-2^{-2n-8}\sum_{i=0}^{n-2}2^{-2i}\sum_{k=1}^{i}a_k2^{2k}.\end{align*}
\end{lemma}

\begin{proof}
  We write $S(m_2):=\sum_{k=1}^{j_2}\frac{s_k \oplus \sigma_k+s_k \oplus \sigma_k'}{2^{n+1-k}}$. Then we have
	\begin{align*}
	  \sum_{m_2\in\DD_{j_2}}\mu_{\bsj,\bsm}^2=& \sum_{s_1,\dots,s_{j_2}=0}^1 \bigg\{(2^{-2n-2}-2^{-n-j_2-3}-2^{-2n-1}(a_{j_2+1}\sigma_{j_2+1}\oplus\sigma_n))^2 \\
		     &+ 2^{-2j_2-2}(2^{-2n-2}-2^{-n-j_2-3}-2^{-2n-1}(a_{j_2+1}\sigma_{j_2+1}\oplus\sigma_n))S(m_2)\\&+2^{-4j_2-6}S(m_2)^2\bigg\}.
	\end{align*}
	Since
	\begin{align*}
	  \sum_{m_2\in\DD_{j_2}}S(m_2)=&\sum_{s_1,\dots,s_{j_2}=0}^1 \sum_{k=1}^{j_2}\frac{s_k \oplus \sigma_k+s_k \oplus \sigma_k'}{2^{n+1-k}}\\
		 =&\sum_{k=1}^{j_2}2^{j_2-1}\sum_{s_k=0}^1\frac{s_k \oplus \sigma_k+s_k \oplus \sigma_k\oplus a_k}{2^{n+1-k}}  \\
		=&\sum_{k=1}^{j_2}2^{j_2-1}\frac{2}{2^{n+1-k}}=2^{2j_2-n}-2^{j_2-n}
	\end{align*}
	and
	\begin{align*}
	  \sum_{m_2\in\DD_{j_2}}S(m_2)^2=&\sum_{s_1,\dots,s_{j_2}=0}^1 \bigg\{\sum_{\substack{k_1,k_2=1\\ k_1\neq k_2}}^{j_2}\frac{(s_{k_1} \oplus \sigma_{k_1}+s_{k_1} \oplus \sigma_{k_1}')(s_{k_2} \oplus \sigma_{k_2}+s_{k_2} \oplus \sigma_{k_2}')}{2^{n+1-k_1}2^{n+1-k_2}} \\ &\hspace{3cm}+\sum_{k=1}^{j_2}\frac{(s_k \oplus \sigma_k+s_k \oplus \sigma_k')^2}{2^{2n+2-2k}} \bigg\} \\
		=& \sum_{\substack{k_1,k_2=1\\ k_1\neq k_2}}^{j_2}2^{j_2-2}\frac{4}{2^{n+1-k_1}2^{n+1-k_2}}+\sum_{k=1}^{j_2}2^{j_2-1}\frac{a_k^2+(1+a_k\oplus 1)^2}{2^{2n+2-2k}} \\
		=&\frac13 2^{-2n+j_2+2}+\frac13 2^{-2n+3j_2+1}-2^{-2n+2j_2+1}+\sum_{k=1}^{j_2}2^{j_2-1}\frac{4-2a_k}{2^{2n+2-2k}},
	\end{align*}
	we obtain the claimed result by combining all these expressions, summing $2^{|\bsj|}\sum_{\bsm\in\DD_{\bsj}}$ over all $\bsj\in\mathcal{J}_2$ and using the fact that
	\begin{align*}
	  \sum_{j_2=0}^{n-2}\sigma_{j_2+1}\oplus a_{j_2+1}\sigma_n=&\sum_{i=1}^{n-1}\sigma_{i}\oplus a_{i}\sigma_n=\sum_{i=1}^{n-1}(\sigma_i-a_i\sigma_n)^2 \\
		   =& \sum_{i=1}^{n-1}(\sigma_i-2a_i\sigma_i\sigma_n+a_i\sigma_n)=\sum_{i=1}^{n-1}\sigma_i+\sigma_nL.
	\end{align*}
\end{proof}

\paragraph{Case 3: $\bsj\in\mathcal{J}_3:=\{(-1,n-1)\}$}

\begin{proposition}
  Let $\bsj\in \mathcal{J}_3$ and $\bsm\in \DD_{\bsj}$. Then we have
     $$ \mu_{\bsj,\bsm}=2^{-2n-1}\left(-\sigma_n+\sum_{k=1}^{n-1}\frac{s_k \oplus a_k(\sigma_n\oplus 1)\oplus \sigma_k}{2^{n-k}}\right). $$
\end{proposition}

\begin{proof}
For $j_2=n-1$ we have $1-|2m_2+1-2^{j_2+1}z_2|=1-|1-b_n|=b_n=t_n\oplus \sigma_n$.
  Writing $\varepsilon(t_n,m_2):=\sum_{k=1}^{n-1}\frac{s_k \oplus a_kt_n \oplus \sigma_k}{2^{n+1-k}}$, we get
	\begin{align*} \sum_{\bsz\in I_{\bsj,\bsm}}(1-z_1)(1-|2m_2+1-2^{j_2+1}|)=&\sum_{t_n=0}^{1} \left(1-\frac{t_n}{2}-\varepsilon(t_n,m_2)\right)(t_n\oplus \sigma_n) \\
	                       =& 1-\frac{\sigma_n\oplus 1}{2}-\varepsilon(\sigma_n\oplus 1,m_1),
	\end{align*}
	which leads to $\mu_{\bsj,\bsm}=2^{-2n-1}(\sigma_n \oplus 1+2\varepsilon(\sigma_n\oplus 1,m_1)-1)$ via~\eqref{art3} and hence to the result.
\end{proof}

\begin{lemma}
 We have 
 $$\sum_{\bsj\in \mathcal{J}_3}2^{|\bsj|}\sum_{\bsm\in\DD_{\bsj}}|\mu_{\bsj,\bsm}|^2=\frac13 2^{-4n-4}(2^{2n}-3\cdot 2^n +2+3\sigma_n 2^{n+1}).$$
\end{lemma}

\begin{proof}
  We have
	\begin{align*}
	   \sum_{\bsj\in \mathcal{J}_3}2^{|\bsj|}\sum_{\bsm\in\DD_{\bsj}}|\mu_{\bsj,\bsm}|^2 =& 2^{n-1}\sum_{s_1,\dots,s_{n-1}=0}^1\left(2^{-2n-1}\left(-\sigma_n+\sum_{k=1}^{n-1}\frac{s_k \oplus a_k(\sigma_n\oplus 1)\oplus \sigma_k}{2^{n-k}}\right)\right)^2 \\
		=& 2^{n-1}\sum_{l=0}^{2^{n-1}-1}\left(2^{-2n-1}\left(-\sigma_n+\frac{l}{2^{n-1}}\right)\right)^2,
	\end{align*}
	which leads to the claimed result.
	\end{proof}

\paragraph{Case 4: $\bsj\in\mathcal{J}_4:=\{(-1,j_2): j_2\geq n\}$}

\begin{proposition}
  Let $\bsj\in \mathcal{J}_4$ and $\bsm\in \DD_{\bsj}$. Then we have
     $$ \mu_{\bsj,\bsm}=2^{-2j_2-3}. $$
\end{proposition}

\begin{proof}
 If $j_2\geq n$, no point of $\cP$ is contained in the interior of $I_{\bsj,\bsm}$ and therefore only the linear part $-t_1t_2$ contributes to the Haar coefficient of the discrepancy function
in this case. Hence, the given formula is an immediate consequence of~\eqref{art3}.
\end{proof}

\begin{lemma}
 We have 
 $$\sum_{\bsj\in \mathcal{J}_4}2^{|\bsj|}\sum_{\bsm\in\DD_{\bsj}}|\mu_{\bsj,\bsm}|^2=\frac{1}{48}2^{-2n}.$$
\end{lemma}
  
\begin{proof}
 It is easy to compute
   $$\sum_{\bsj\in \mathcal{J}_4}2^{|\bsj|}\sum_{\bsm\in\DD_{\bsj}}|\mu_{\bsj,\bsm}|^2=\sum_{j_2=n}^{\infty} 2^{2j_2} 2^{-4j_2-6}=\frac{1}{48}2^{-2n}.$$
\end{proof}

\paragraph{Case 5: $\bsj\in\mathcal{J}_5:=\{(0,-1)\}$}

\begin{proposition} \label{prop5}
  Let $\bsj\in \mathcal{J}_5$ and $\bsm\in \DD_{\bsj}$. Then we have
     $$ \mu_{\bsj,\bsm}=\frac{1}{2^{2n+2}}-\frac{1}{2^{n+3}}+\frac{1}{2^{n+3}}L-\frac{1}{2^{2n+1}}\sigma_n. $$
\end{proposition}

\begin{proof}
  For $\bsz\in\cP_{\bsa}(\vecs)$ we have 
	$$ 1-z_2=1-\frac{b_1}{2}-\dots-\frac{b_n}{2^n}=1-\frac{t_1\oplus a_1t_n \oplus \sigma_1}{2}-\dots-\frac{t_{n-1}\oplus a_{n-1}t_n \oplus \sigma_{n-1}}{2^{n-1}}-\frac{t_n \oplus \sigma_n}{2^n} $$
	and
	$$ 1-|2m_1+1-2z_1|=1-|1-2z_1|=1-\left|1-t_n-\frac{t_{n-1}}{2}-\dots-\frac{t_1}{2^{n-1}}\right|. $$
	Hence, writing $u=2^{-1}t_{n-1}+\dots+2^{-n+1}t_1$, $v_1=2^{-1}(t_1\oplus\sigma_1)+\dots+2^{-n+1}(t_{n-1}\oplus\sigma_{n-1})$ and $v_2=2^{-1}(t_1\oplus\sigma_1')+\dots+2^{-n+1}(t_{n-1}\oplus\sigma_{n-1}')$, we have
	\begin{align*}
	  &\sum_{\bsz\in\cP}(1-|2m_1+1-2z_1|)(1-z_2) \\ =& \sum_{t_1,\dots,t_{n}=0}^1 \left(1-\left|1-t_n-\frac{t_{n-1}}{2}-\dots-\frac{t_1}{2^{n-1}}\right|\right) \\ &\times \left(1-\frac{t_1\oplus a_1t_n \oplus \sigma_1}{2}-\dots-\frac{t_{n-1}\oplus a_{n-1}t_n \oplus \sigma_{n-1}}{2^{n-1}}-\frac{t_n \oplus \sigma_n}{2^n}\right) \\
		 =& \sum_{t_1,\dots,t_{n-1}=0}^1 \left\{u\left(1-v_1-\frac{\sigma_n}{2^n}\right)+(1-u)\left(1-v_2-\frac{\sigma_n\oplus 1}{2^n}\right)\right\} \\
		 =& \sum_{t_1,\dots,t_{n-1}=0}^1 \left\{1-2^{-n}+2^{-n}\sigma_n-v_2+u(2^{-n}-2^{-n+1}\sigma_n)+u(v_2-v_1)\right\} \\
		 =& 2^{n-1}(1-2^{-n}+2^{-n}\sigma_n) +(2^{n-2}-2^{-1})(-1+2^{-n}-2^{-n+1}\sigma_n)+\sum_{t_1,\dots,t_{n-1}=0}^1u(v_2-v_1),
	\end{align*}
	where we regarded $\sum_{t_1,\dots,t_{n-1}=0}^1u=\sum_{t_1,\dots,t_{n-1}=0}^1v_2=2^{n-2}-2^{-1}$ in the last step. By~\eqref{formelv1} and~\eqref{formelv2} we find
	$$ \sum_{t_1,\dots,t_{n-1}=0}^1u(v_2-v_1)=\frac14 \sum_{k=1}^{n-1}\left(\sigma_k'\oplus 1 - \sigma_k \oplus 1\right)=-\frac14\sum_{k=1}^{n-1}a_k(1-2\sigma_k)=-\frac{L}{4}, $$
	and therefore
	$$ \sum_{\bsz\in\cP}(1-|2m_1+1-2z_1|)(1-z_2)=\frac14-2^{-n-1}+2^{n-2}+2^{-n}\sigma_n-\frac{L}{4}. $$
	The rest follows with~\eqref{art2}.
\end{proof}

\begin{lemma}
 We have 
 $$\sum_{\bsj\in \mathcal{J}_5}2^{|\bsj|}\sum_{\bsm\in\DD_{\bsj}}|\mu_{\bsj,\bsm}|^2=\left(\frac{1}{2^{2n+2}}-\frac{1}{2^{n+3}}+\frac{1}{2^{n+3}}L-\frac{1}{2^{2n+1}}\sigma_n\right)^2.$$
\end{lemma}

\paragraph{Case 6: $\bsj\in\mathcal{J}_6:=\{(j_1,-1): 1\leq j_1 \leq n-1 \}$}

\begin{proposition}
  Let $\bsj\in \mathcal{J}_6$ and $\bsm\in \DD_{\bsj}$. Then we have
     $$ \mu_{\bsj,\bsm}=2^{-2n-2j_1-3}\left(2^{2j_1+1}-2^{j_1+n}+2^{2n+1}\varepsilon(m_1)-2^{2j_1+2}(a_{n-j_1}r_1 \oplus \sigma_{n-j_1})\right), $$
     where $\varepsilon(m_1)=\frac{r_1 \oplus \sigma_n}{2^n}+\sum_{k=2}^{j_1}\frac{r_k \oplus a_{n+1-k}r_1\oplus \sigma_{n+1-k}}{2^{n+1-k}}$.
\end{proposition}

\begin{proof}
 Similar to the proof of Proposition 2 we write
$$ 1-z_2=1-u-\frac{b_{n-j_1}}{2^{n-j_1}}-\varepsilon(m_1) $$
with $u=2^{-1}b_1+\dots+2^{-n+j_1+1}b_{n-j_1-1}$ and 
$$\varepsilon(m_1)=2^{-n+j_1-1}b_{n-j_1+1}+\dots+2^{-n}b_n=\frac{r_1 \oplus \sigma_n}{2^n}+\sum_{k=2}^{j_1}\frac{r_k \oplus a_{n+1-k}r_1\oplus \sigma_{n+1-k}}{2^{n+1-k}}.$$
We also have
\begin{align*}
	 1-|2m_1+1-2^{j_1+1}z_1|=& 1-|1-t_{n-j_1}-\dots-2^{j_1-n+1}t_1|\\ =&\begin{cases}
	                                                                    v & \text{if \,} t_{n-j_1}=0, \\
																																			1-v & \text{if \,} t_{n-j_1}=1,
	                                                                \end{cases}
		\end{align*}
		where $v=2^{-1}t_{n-j_1-1}+\dots+2^{j_1-n+1}t_1$.
		Note that $\varepsilon$ and $v$ are fully determined by the condition $\bsz\in I_{\bsj,\bsm}$, as this condition
		fixes $t_n=r_1$ in particular. The only free digits are $t_1,\dots,t_{n-j_1}$. Let us first fix $t_1,\dots,t_{n-j_1-1}$
		and hence $u$. Then we have
		\begin{align*}
		  \sum_{t_{n-j_1}=0}^{1}&\left(1-u-\frac{b_{n-j_1}}{2^{n-j_1}}-\varepsilon(m_1)\right)(1-|1-t_{n-j_1}-\dots-2^{j_1-n+1}t_1|)\\
			 =& \left(1-u-\frac{a_{n-j_1}r_1\oplus \sigma_{n-j_1}}{2^{n-j_1}}-\varepsilon(m_1)\right)v \\&+\left(1-u-\frac{1\oplus a_{n-j_1}r_1\oplus \sigma_{n-j_1}}{2^{n-j_1}}-\varepsilon(m_1)\right)(1-v) \\
			 =& 1-2^{-n+j_1}-\varepsilon-u+2^{-n+j_1}v-2^{-n+j_1+1}(a_{n-j_1}r_1\oplus\sigma_{n-j_1})v\\&+2^{-n+j_1}(a_{n-j_1}r_1\oplus\sigma_{n-j_1}).
		\end{align*}
		Regarding $\sum_{t_1,\dots,t_{n-j_1-1}=0}^{1}u=\sum_{t_1,\dots,t_{n-j_1-1}=0}^{1}v=\sum_{l=0}^{2^{n-j_1-1}-1}\frac{l}{2^{n-j_1-1}}=2^{n-j_1-2}-\frac12$, we find
		\begin{align*} \sum_{\bsz\in I_{\bsj,\bsm}}&(1-z_2)(1-|2m_1+1-2^{j_1+1}z_1|) \\ =&2^{-n-j_1-2}(2^{2n}+2^{j_1+n}-2^{2j_1+1}-2^{2n+1}\varepsilon(m_1)+2^{2j_1+2}(a_{n-j_1}r_1\oplus \sigma_{n-j_1})). \end{align*}
		The rest follows with~\eqref{art2}.
\end{proof}

\begin{lemma}
 We have 
 $$\sum_{\bsj\in \mathcal{J}_6}2^{|\bsj|}\sum_{\bsm\in\DD_{\bsj}}|\mu_{\bsj,\bsm}|^2=\frac19 4^{-2n-3}((3n+11) 4^n-56)-2^{-3n-4}\left(n-1-2\sum_{i=1}^{n-1}\sigma_i-2\sigma_n L\right).$$
\end{lemma}

\begin{proof}
 To compute $\sum_{m_2\in\DD_{j_2}}\mu_{\bsj,\bsm}^2$, we first sum over $r_1$, write $\varepsilon=\varepsilon(r_1)$ and find
\begin{align*}
  \sum_{r_2,\dots,r_{j_1}}\mu_{\bsj,\bsm}^2=&\big\{ (2^{-2n-2j_1-3}(2^{2j_1+1}-2^{j_1+n}+2^{2n+1}\varepsilon(0)-2^{2j_1+2}\sigma_{n-j_1}))^2 \\
	&+(2^{-2n-2j_1-3}(2^{2j_1+1}-2^{j_1+n}+2^{2n+1}\varepsilon(1)-2^{2j_1+2}\sigma_{n-j_1}'))^2	\big\}.
\end{align*}
We arrive at the claimed formula by writing $\varepsilon(0)=\frac{l}{2^{n-1}}+\frac{\sigma_n}{2^n}$ and $\varepsilon(1)=\frac{l}{2^{n-1}}+\frac{1-\sigma_n}{2^n}$ 
and by replacing the sum over $r_2,\dots, r_{j_1}$ by a sum over $l$ running from $0$ to $2^{j_1-1}-1$.
\end{proof}

\paragraph{Case 7: $\bsj\in\mathcal{J}_7:=\{(j_1,-1): j_1\geq n\}$}

This case is completely analogous to Case 4.

\begin{proposition}
  Let $\bsj\in \mathcal{J}_7$ and $\bsm\in \DD_{\bsj}$. Then we have
     $$ \mu_{\bsj,\bsm}=-2^{-2j_1-3}. $$
\end{proposition}

\begin{lemma}
 We have 
 $$\sum_{\bsj\in \mathcal{J}_7}2^{|\bsj|}\sum_{\bsm\in\DD_{\bsj}}|\mu_{\bsj,\bsm}|^2=\frac{1}{48}2^{-2n}.$$
\end{lemma}

\paragraph{Case 8: $\bsj\in\mathcal{J}_8:=\{(0,j_2): 0\leq j_2\leq n-2\}$}

\begin{proposition}
  Let $\bsj\in \mathcal{J}_8$ and $\bsm\in \DD_{\bsj}$. Then we have
     $$ \mu_{\bsj,\bsm}=2^{-2j_2-4}\sum_{k=1}^{j_2}a_k\frac{2s_k\oplus \sigma_k-1}{2^{n-k}}+2^{-2n-2}\left(1+2\sigma_{j_2+1}(\sigma_n-1)+2\sigma_n(\sigma_{j_2+1}'-1)\right). $$
\end{proposition}

\begin{proof}
  In this case, the condition $\bsz\in I_{\bsj,\bsm}$ results in $b_k=s_k$ for all $k\in\{1,\dots,j_2\}$ and
	$$ 2m_1+1-2^{j_1+1}z_1=1-2z_1=1-t_n-\dots-2^{-n+1}t_1=1-t_n-u-2^{-n+j_2+1}t_{j_2+1}-\varepsilon(m_2,t_n), $$
	where $u=2^{-1}t_{n-1}+\dots+2^{-n+j_2+2}t_{j_2+2}$ and $\varepsilon(m_2,t_n)=\sum_{k=1}^{j_2}\frac{s_k\oplus a_k t_n \oplus \sigma_k}{2^{n-k}}$.
	Further we have
	$$ 2m_2+1-2^{j_2+1}z_2=1-b_{j_2+1}-\dots-2^{j_2-n+1}b_n=1-t_{j_2+1}\oplus a_{j_2+1}t_n\oplus \sigma_{j_2+1}-v(t_n)-2^{-n+j_2+1}t_n \oplus \sigma_n $$
	with $v(t_n)=\sum_{k=j_2}^{n-1}2^{-k+n}(s_k \oplus a_k t_n \oplus \sigma_k)$. Therefore
	\begin{align*}
	  \sum_{\bsz\in I_{\bsj,\bsm}} &(1-|2m_1+1-2^{j_1+1}z_1|)(1-|2m_2+1-2^{j_2+1}z_2|)\\
		 =&\sum_{t_{j_2+1},\dots,t_n=0}^1(1-|1-t_n-u-2^{-n+j_2+1}t_{j_2+1}+\varepsilon(m_2,t_n)|) \\
		     &\times(1-|1-t_{j_2+1}\oplus a_{j_2+1}t_n\oplus \sigma_{j_2+1}-v(t_n)-2^{-n+j_2+1}t_n \oplus \sigma_n|) \\
		=&\sum_{t_{j_2+2},\dots,t_{n-1}=0}^1\bigg\{\left(u+2^{-n+j_2+1}\sigma_{j_2+1}+\varepsilon(m_2,0)\right)\left(v(0)+2^{-n+j_2++1}\sigma_n\right) \\
		 &+\left(u+2^{-n+j_2+1}(\sigma_{j_2+1}\oplus 1)+\varepsilon(m_2,0)\right)\left(1-v(0)-2^{-n+j_2+1}\sigma_n\right) \\
		&+\left(1-u-2^{-n+j_2+1}\sigma_{j_2+1}'-\varepsilon(m_2,1)\right)\left(v(1)+2^{j_2-n+1}(\sigma_n \oplus 1)\right) \\
		&+\left(1-u-2^{-n+j_2+1}(\sigma_{j_2+1}'\oplus 1)-\varepsilon(m_2,1)\right)\left(1-v(1)-2^{-n+j_2+1}(\sigma_n \oplus 1)\right)\bigg\} \\
		=& \sum_{t_{j_2+2},\dots,t_{n-1}=0}^1 4^{-n}\bigg\{4^{j_2+1}(2\sigma_n(\sigma_{j_2+1}+\sigma_{j_2+1}'-1)-2\sigma_{j_2+1}'+1) \\
		 &\hspace{2cm}+4^n(1+\varepsilon(m_2,0)-\varepsilon(m_2,1)+2^{n+j_2+1}(\sigma_{j_2+1}'-\sigma_{j_2+1})\\
		 &\hspace{2cm} +2^{n+j_2+1}((2\sigma_{j_2+1}-1)v(0)-(2\sigma_{j_2+1}'-1)v(1))\bigg\}.
	\end{align*}
	In the last expression, only $v(0)$ and $v(1)$ depend on the digits $t_{j_2+2},\dots,t_{n-1}$ and we have
	$$ \sum_{t_{j_2+2},\dots,t_{n-1}=0}^1v(0)=\sum_{t_{j_2+2},\dots,t_{n-1}=0}^1 v(1)=\sum_{l=0}^{2^{n-j_2-2}-1}\frac{l}{2^{n-j_2-2}}=2^{n-j_2-3}-\frac12. $$
	Hence, we can compute $\sum_{\bsz\in I_{\bsj,\bsm}} (1-|2m_1+1-2^{j_1+1}z_1|)(1-|2m_2+1-2^{j_2+1}z_2|)$ and the Haar coefficients via~\eqref{art4}. Note that
	$$ \varepsilon(m_2,0)-\varepsilon(m_2,1)=\sum_{k=1}^{j_2}\frac{s_k \oplus \sigma_k-s_k \oplus a_k \oplus \sigma_k}{2^{n-k}}=\sum_{k=1}^{j_2}a_k\frac{2s_k\oplus \sigma_k-1}{2^{n-k}}, $$
	where the relation $s_k \oplus \sigma_k-s_k \oplus a_k \oplus \sigma_k=a_k (2(s_k\oplus \sigma_k)-1)$ can be seen easily.
\end{proof}

\begin{lemma}
 We have 
 $$\sum_{\bsj\in \mathcal{J}_8}2^{|\bsj|}\sum_{\bsm\in\DD_{\bsj}}|\mu_{\bsj,\bsm}|^2=\frac13 4^{-2n-3}(4^n-4)+2^{-2n-8}\sum_{i=0}^{n-2}2^{-2i}\sum_{k=1}^{i}a_k2^{2k}.$$
\end{lemma}

\begin{proof}
  For the sake of brevity we write $$f:=2^{-2n-2}\left(1+2\sigma_{j_2+1}(\sigma_n-1)+2\sigma_n(\sigma_{j_2+1}'-1)\right).$$ Note that $f$ does not depend on $m_2$. Then $\sum_{m_2\in\DD_{\bsj_2}}\mu_{\bsj,\bsm}^2$ equals
	$$ \sum_{m_2\in\DD_{\bsj_2}}\left\{ 2^{-4j_2-8}\left(\sum_{k=1}^{j_2}a_k\frac{2(s_k\oplus \sigma_k)-1}{2^{n-k}}\right)^2+ 2^{-2j_2-3}f\sum_{k=1}^{j_2}a_k\frac{2(s_k\oplus \sigma_k)-1}{2^{n-k}}+f^2 \right\}. $$
	Since
	$$ \sum_{s_1,\dots,s_{j_2}=0}^1\sum_{k=1}^{j_2}a_k\frac{2(s_k\oplus \sigma_k)-1}{2^{n-k}}=\sum_{k=1}^{j_2}\frac{a_k}{2^{n-k}}2^{j_2-1}\sum_{s_k=0}^1 (2(s_k \oplus \sigma_k)-1)=0 $$
	and
	\begin{align*}
	  \sum_{s_1,\dots,s_{j_2}=0}^1&\left(\sum_{k=1}^{j_2}a_k\frac{2(s_k\oplus \sigma_k)-1}{2^{n-k}}\right)^2 \\
		     =&\sum_{s_1,\dots,s_{j_2}=0}^1 \bigg( \sum_{\substack{k_1,k_2=1\\ k_1\neq k_2}}^{j_2}a_{k_1}a_{k_2}\frac{(2(s_{k_1}\oplus\sigma_{k_1})-1)(2(s_{k_2}\oplus\sigma_{k_2})-1)}{2^{n-k_1}2^{n-k_2}}
				 \\ &+\sum_{k=1}^{j_2}\frac{(2(s_k\oplus\sigma_k)-1)^2}{2^{2n-2k}} \bigg) \\
				=& \sum_{\substack{k_1,k_2=1\\ k_1\neq k_2}}^{j_2}\frac{a_{k_1}a_{k_2}}{2^{n-k_1}2^{n-k_2}}2^{j_2-2}\sum_{s_{k_1},s_{k_2}=0}^1(2(s_{k_1}\oplus\sigma_{k_1})-1)(2(s_{k_2}\oplus\sigma_{k_2})-1) \\
				&+ \sum_{k=1}^{j_2} \frac{a_k}{2^{2n-2k}}2^{j_2-1}\sum_{s_k=0}^1 (2(s_k \oplus \sigma_k)-1)^2=2^{j_2-2n}\sum_{k=1}^{j_2}a_k2^{2k}.
	\end{align*}
	This yields
	$$ \sum_{m_2\in\DD_{\bsj_2}}\mu_{\bsj,\bsm}^2=2^{-3j_2-2n-8}\sum_{k=1}^{j_2}a_k2^{2k}+2^{j_2}f^2. $$
	Note that $1+2\sigma_{j_2+1}(\sigma_n-1)+2\sigma_n(\sigma_{j_2+1}'-1)=1-2\sigma_{j_2+1}$ if $\sigma_n=0$ and $1+2\sigma_{j_2+1}(\sigma_n-1)+2\sigma_n(\sigma_{j_2+1}'-1)=2\sigma_{j_2+1}'-1$ if $\sigma_n=1$; thus $(1+2\sigma_{j_2+1}(\sigma_n-1)+2\sigma_n(\sigma_{j_2+1}'-1))^2=1$ in any case and $f^2=2^{-4n-4}$. After summation over $j_2$ we obtain the result.
\end{proof}

\paragraph{Case 9: $\bsj\in\mathcal{J}_{9}:=\{(j_1,j_2)\in \NN_0^2: j_1+j_2\leq n-2, j_1\geq 1 \}$}

\begin{proposition}
  Let $\bsj\in \mathcal{J}_{9}$ and $\bsm\in \DD_{\bsj}$. Then we have
     $$ \mu_{\bsj,\bsm}=2^{-2n-2}(2(a_{n-j_1}r_1\oplus \sigma_{j_2+1})-1)(2(a_{j_2+1}r_1\oplus \sigma_{n-j_1})-1). $$
\end{proposition}

\begin{proof}
   The condition $\bsz\in I_{\bsj,\bsm}$ yields, by~\eqref{cond},~\eqref{z1} and~\eqref{z2}, that
	$$ 2m_1+1-2^{j_1+1}z_1=1-t_{n-j_1}-u-2^{j_1+j_2-n+1}t_{j_2+1}-\varepsilon_1 $$
	with $u=2^{-1}t_{n-j_1-1}+\dots+2^{j_1+j_2-n}t_{j_2+1}$ and 
	$$ \varepsilon_1=2^{j_1-n+1}\sum_{k=1}^{j_2} 2^{k-1}t_k=2^{j_1-n+1}\sum_{k=1}^{j_2} 2^{k-1}(s_k\oplus a_kr_1 \oplus \sigma_k)=\varepsilon_1(\bsm). $$
	Similarly, we write
	$$ 2m_2+1-2^{j_2+1}z_2=1-b_{j_2+1}-v-2^{j_1+j_2-n+1}b_{n-j_1}-\varepsilon_2 $$
	with $v=2^{-1}b_{j_2+2}+\dots+2^{j_1+j_2-n}b_{n-j_1-1}$ and
	$$ \varepsilon_2=2^{j_2-n+1}\sum_{k=1}^{j_1} 2^{k-1}b_{n+1-k}=2^{j_2-n+1}\left(r_1\oplus \sigma_n+\sum_{k=1}^{j_2} 2^{k-1}(r_k\oplus a_{n+1-k}r_1 \oplus \sigma_{n+1-k})\right)=\varepsilon_2(\bsm). $$	
	We fix the digits $t_{j_2+2},\dots,t_{n-j_1-1}$; then $u$ and $v$ are also fixed. We sum $$(1-|2m_1+1-2^{j_1+1}z_1|)(1-|2m_2+1-2^{j_2+1}z_2|)$$
	over $t_{n-j_1}\in\{0,1\}$ and $t_{j_2+2}\in\{0,1\}=\{a_{j_2+1}r_1\oplus \sigma_{j_2+1},a_{j_2+1}r_1\oplus \sigma_{j_2+1}\oplus 1\}$ and  find after lengthy calculations
	\begin{align*} \sum_{t_{j_2+1},t_{n-j_1}=0}^1& (1-|1-t_{n-j_1}-u-2^{j_1+j_2-n+1}t_{j_2+1}-\varepsilon_1|)\\ &\times(1-|1-b_{j_2+1}-v-2^{j_1+j_2-n+1}b_{n-j_1}-\varepsilon_2|)  \\
	   =& 1+4^{-n+j_1+j_2+1} (2(a_{n-j_1}r_1\oplus \sigma_{j_2+1})-1)(2(a_{j_2+1}r_1\oplus \sigma_{n-j_1})-1).
	\end{align*}
	Summation over the remaining digits $t_{j_2+2},\dots,t_{n-j_1-1}$ yields
	\begin{align*} \sum_{\bsz\in I_{\bsj,\bsm}}&(1-|2m_1+1-2^{j_1+1}z_1|)(1-|2m_2+1-2^{j_2+1}z_2|) \\
	   =&2^{n-j_1-j_2-2}+2^{-n+j_1+j_2}(2(a_{n-j_1}r_1\oplus \sigma_{j_2+1})-1)(2(a_{j_2+1}r_1\oplus \sigma_{n-j_1})-1), \end{align*}
	and the result follows by~\eqref{art4}.
\end{proof}
Since $\mu_{\bsj,\bsm}^2=2^{-4n-4}$ is independent of $\bsj$ and $\bsm$ in this case, the following consequence is straightforward.
\begin{lemma}
 We have 
 $$\sum_{\bsj\in \mathcal{J}_{9}}2^{|\bsj|}\sum_{\bsm\in\DD_{\bsj}}|\mu_{\bsj,\bsm}|^2=\frac19 4^{-2n-3} \left(3n 4^n-7\cdot 4^n+16\right).$$
\end{lemma}

\paragraph{Case 10: $\bsj\in\mathcal{J}_{10}:=\{(0,n-1)\}$}

\begin{proposition}
  Let $\bsj\in \mathcal{J}_{10}$ and $\bsm\in \DD_{\bsj}$. Then we have
     \begin{align*} \mu_{\bsj,\bsm}=&\frac{1}{2^{2n+2}}\left(1-2\left|\sigma_n-\sum_{k=1}^{n-1}\frac{s_k \oplus a_k(\sigma_n\oplus 1) \oplus \sigma_k}{2^{n-k}}\right|\right).
		 \end{align*}
\end{proposition}

\begin{proof}
  We have $1-|2m_1+1-2z_1|=1-\left|1-t_n-\sum_{k=1}^{n-1}\frac{s_k\oplus a_kt_n \oplus \sigma_k}{2^{n-k}}\right|$ and $1-|2m_2+1-2^{n}z_2|=1-|1-b_n|=b_n$ and 
	therefore
	\begin{align*} \sum_{\bsz\in I_{\bsj,\bsm}}&(1-|2m_1+1-2z_1|)(1-|2m_2+1-2^{n}z_2|) \\
	   =&\sum_{t_n=0}^{1}\left(1-\left|1-t_n-\sum_{k=1}^{n-1}\frac{s_k\oplus a_kt_n \oplus \sigma_k}{2^{n-k}}\right|\right)(t_n \oplus \sigma_n) \\
		 =& 1-\left|1-\sigma_n\oplus 1-\sum_{k=1}^{n-1}\frac{s_k\oplus a_k(\sigma_n \oplus 1) \oplus \sigma_k}{2^{n-k}}\right|;
	\end{align*}
	the rest of the proof is straightforward by~\eqref{art4}.
\end{proof}

\begin{lemma}
 We have 
 $$\sum_{\bsj\in \mathcal{J}_{10}}2^{|\bsj|}\sum_{\bsm\in\DD_{\bsj}}|\mu_{\bsj,\bsm}|^2=\frac{1}{3}2^{-4n-6}\left(2^{2n}+8\right).$$
\end{lemma}

\begin{proof}
   In both cases $\sigma_n=0$ and $\sigma_n=1$ we find
	  $$ 2^{n-1}\sum_{m_2=0}^{2^{n-1}-1}\mu_{\bsj,\bsm}^2=2^{n-1}\frac{1}{2^{4n+4}}\sum_{l=0}^{2^{n-1}-1}\left(1-2\frac{l}{2^{n-1}}\right)^2, $$
		which yields the claim.
\end{proof}

\paragraph{Case 11: $\bsj\in\mathcal{J}_{11}:=\{(j_1,j_2)\in \NN_0^2: j_1+j_2= n-1, j_1\geq 1\}$}

\begin{proposition}
  Let $\bsj\in \mathcal{J}_{11}$ and $\bsm\in \DD_{\bsj}$. Then we have 
     \begin{align*} \mu_{\bsj,\bsm}=&2^{-2n-1}\Bigg\{\left(1-\left|1-a_{j_2+1}r_1\oplus \sigma_{j_2+1}-\sum_{k=2}^{j_1}\frac{r_k\oplus a_{n+1-k}r_1\oplus \sigma_{n+1-k}}{2^{j_1-k+1}}-\frac{r_1\oplus\sigma_n}{2^{j_1}}\right|\right) \\
		       &\times \left(\sum_{k=1}^{j_2}\frac{s_k\oplus a_kr_1 \oplus \sigma_k}{2^{j_2-k+1}}\right) \\
					&+ \left(1-\left|1-1\oplus a_{j_2+1}r_1\oplus \sigma_{j_2+1}-\sum_{k=2}^{j_1}\frac{r_k\oplus a_{n+1-k}r_1\oplus \sigma_{n+1-k}}{2^{j_1-k+1}}-\frac{r_1\oplus\sigma_n}{2^{j_1}}\right|\right)\\
		      &\times \left(1-\sum_{k=1}^{j_2}2^{k-1-j_2}\frac{s_k\oplus a_kr_1 \oplus \sigma_k}{2^{j_2-k+1}}\right)\Bigg\}-2^{-2n-2}.
		 \end{align*}
\end{proposition}

\begin{proof}
 By the condition $\bsz\in I_{\bsj,\bsm}$ all digits but $t_{j_2+1}=t_{n-j_1}$ are fixed. Hence, we get the result by summing $(1-|2m_1+1-2^{j_1+1}z_1|)(1-|2m_2+1-2^{j_2+1}z_2|)$
over the two possibilities $t_{j_2+1}=0,1$ and expressing the other digits of $z_1$ and $z_2$ in terms of the digits $r_i$ and $s_j$ of $m_1$ and $m_2$ according to~\eqref{cond}.
\end{proof}

\begin{lemma}
 We have 
 $$\sum_{\bsj\in \mathcal{J}_{11}}2^{|\bsj|}\sum_{\bsm\in\DD_{\bsj}}|\mu_{\bsj,\bsm}|^2=\frac{1}{27}2^{-4n-6}\left(3n2^{2n}+7\cdot2^{2n}+48n-88\right).$$
\end{lemma}

\begin{proof}
As always, we first investigate $\sum_{\bsm\in\DD_{\bsj}}\mu_{\bsj,\bsm}^2$. We sum over $r_1$ to obtain
\begin{align*}
  \sum_{r_2,\dots,r_{j_1}=0}^{1}&\sum_{s_1,\dots,s_{j_2}=0}^{1} \Bigg\{
	   2^{-2n-1}\left(1-\left|1-1\oplus \sigma_{j_2+1}-\sum_{k=2}^{j_1}2^{k-1-j_1}(r_k\oplus \sigma_{n+1-k})-2^{-j_1}\sigma_n\right|\right) \\
		       &\times \left(\sum_{k=1}^{j_2}2^{k-1-j_2}(s_k \oplus \sigma_k)\right) \\
					&+ 2^{-2n-1}\left(1-\left|1-1\oplus \sigma_{j_2+1}-\sum_{k=2}^{j_1}2^{k-1-j_1}(r_k\oplus \sigma_{n+1-k})-2^{-j_1}\sigma_n\right|\right)\\
		      &\times \left(1-\sum_{k=1}^{j_2}2^{k-1-j_2}(s_k \oplus \sigma_k)\right)-2^{-2n-2}\Bigg\}^2\\
					+ \sum_{r_2,\dots,r_{j_1}=0}^{1}&\sum_{s_1,\dots,s_{j_2}=0}^{1} \\
					&\Bigg\{
	   2^{-2n-1}\left(1-\left|1-1\oplus \sigma_{j_2+1}'-\sum_{k=2}^{j_1}2^{k-1-j_1}(r_k\oplus \sigma_{n+1-k}')-2^{-j_1}(r_1\oplus\sigma_n\oplus 1)\right|\right) \\
		       &\times \left(\sum_{k=1}^{j_2}2^{k-1-j_2}(s_k\oplus a_kr_1 \oplus \sigma_k')\right) \\
					&+ 2^{-2n-1}\left(1-\left|1-1\oplus \sigma_{j_2+1}'-\sum_{k=2}^{j_1}2^{k-1-j_1}(r_k\oplus \sigma_{n+1-k}')-2^{-j_1}(r_1\oplus\sigma_n\oplus 1)\right|\right)\\
		      &\times \left(1-\sum_{k=1}^{j_2}2^{k-1-j_2}(s_k \oplus \sigma_k')\right)-2^{-2n-2}\Bigg\}^2=M_1(\sigma_{j_2+1})+M_2(\sigma_{j_2+1}').
\end{align*}
We can compute $M_1(0)$ via
\begin{align*}
  \sum_{l_1=0}^{2^{j_1-1}-1}\sum_{l_2=0}^{2^{j_2}-1} \Bigg[
	     2^{-2n-1}\bigg\{\frac{l_2}{2^{j_2}}\bigg(\frac{l_1}{2^{j_1-1}}+\frac{\sigma_n}{2^{j_1}}\bigg)+\bigg(1-\frac{l_2}{2^{j_2}}\bigg)\bigg(1-\frac{l_1}{2^{j_1-1}}-\frac{\sigma_n}{2^{j_1}}\bigg)\bigg\}-2^{-2n-2}
	\Bigg]^2.
\end{align*}
Similarly, one calculates $M_1(1)$ and finds $M_1(1)=M_1(0)$. We can compute $M_2(0)$ with the same formula as for $M_1(0)$ - we just have to
replace $\sigma_n$ by $1-\sigma_n$. Again we have $M_2(1)=M_2(0)$ and therefore $\sum_{\bsm\in\DD_{\bsj}}\mu_{\bsj,\bsm}^2=M_1(0)+M_2(0)$.
The rest follows by a straightforward summation of $2^{|\bsj|}\sum_{\bsm\in\DD_{\bsj}}\mu_{\bsj,\bsm}^2$ over all $\bsj\in\mathcal{J}_{11}$.
\end{proof}

\paragraph{Case 12: $\bsj\in\mathcal{J}_{12}:=\{(j_1,j_2)\in \NN_0^2: j_1+j_2\geq  n,  1\leq j_1\leq n-1, 1\leq j_2\leq n-1 \}$}

\begin{proposition}
  Let $\bsj\in \mathcal{J}_{12}$ and $\bsm\in \DD_{\bsj}$. Then we have
     \begin{align*} \mu_{\bsj,\bsm}=&2^{-n-j_1-j_2-2}\left(1-\left|1-\sum_{k=1}^{n-j_1}\frac{s_k \oplus a_kr_1\oplus \sigma_k}{2^{n-j_1-k}}\right|\right) \\
		  &\times  \left(1-\left|1-\sum_{k=2}^{n-j_2}\frac{r_k \oplus a_{n+1-k}r_1\oplus \sigma_{n+1-k}}{2^{n-j_2-k}}-\frac{r_1 \oplus \sigma_n}{2^{n-j_2-1}}\right|\right)-2^{-2j_1-2j_2-4} \end{align*}
		if $s_{\mu}\oplus a_{\mu}r_1 \oplus \sigma_{\mu}=r_{n+1-\mu}$ for all $\mu\in\{n+1-j_1,\dots, j_2\}$, and $\mu_{\bsj,\bsm}=-2^{-2j_1-2j_2-4}$ otherwise.
\end{proposition}

\begin{proof}
  Again, the condition $\bsz\in I_{\bsj,\bsm}$ yields, by~\eqref{cond}, that $t_{n+1-k}=r_k$ for all $k\in \{1,\dots, j_1\}$ and $b_k=s_k$ for all $k\in \{1,\dots, j_2\}$. As a result, for $\mu\in\{n+1-j_1,\dots, j_2\}$
	we must have
	\begin{equation} \label{system} r_{n+1-\mu}=b_\mu \oplus a_{\mu}t_n \oplus \sigma_{\mu}=s_{\mu}\oplus a_{\mu}r_1 \oplus \sigma_{\mu} \end{equation}
	as a condition to have a point of $\cP$ in the dyadic box $I_{\bsj,\bsm}$. Hence, if the system~\eqref{system} of equations is not satisfied, then
	only the linear part of the discrepancy function contributes to the Haar coefficient and hence $\mu_{\bsj,\bsm}=-2^{-2j_1-2j_2-4}$ is this case.
	Assume now that~\eqref{system} is satisfied and let $\bsz=(z_1,z_2)$ be the single point in $I_{\bsj,\bsm}$. Then by ~\eqref{z1} and~\eqref{z2} we obtain
	\begin{align*}
	  \mu_{\bsj,\bsm}=& 2^{-n-j_1-j_2-2}(1-|1-t_{n-j_1}-\dots-2^{j_1-n+1}t_1|)(1-|1-b_{j_2+1}-\dots-2^{j_2-n+1}b_n|) \\ &-2^{-2j_1-2j_2-4},
	\end{align*}
	where the above conditions on the digits give $t_{k}=s_{k}\oplus a_k r_1 \oplus \sigma_k$ for $k=1,\dots,n-j_1$ and $b_{n+1-k}=r_{k}\oplus a_{n+1-k} r_1 \oplus \sigma_{n+1-k}$ for $k=2,\dots,n-j_2$
	as well as $b_{n}=r_1\oplus \sigma_n$. Hence the result follows.
\end{proof}

\begin{lemma}
 We have 
 $$\sum_{\bsj\in \mathcal{J}_{12}}2^{|\bsj|}\sum_{\bsm\in\DD_{\bsj}}|\mu_{\bsj,\bsm}|^2=\frac{1}{27}4^{-2n-2}-\frac{1}{27}4^{-n-2}-\frac19 n 4^{-2n-1}+\frac59 n 4^{-n-3}.$$
\end{lemma}

\begin{proof} We write
 \begin{align*} \sum_{\bsm\in\DD_{\bsj}}|\mu_{\bsj,\bsm}|^2=&\sum_{m_1=0}^{2^{j_1}-1}\left(\sum_{\substack{m_2=0 \\ \eqref{system} \text{\, satisfied}}}^{2^{j_2}-1}\mu_{\bsj,\bsm}^2
	       +\sum_{\substack{m_2=0 \\ \eqref{system} \text{\, not satisfied}}}^{2^{j_2}-1}(-2^{-2j_2-2j_2-4})^2\right) \\
				=& \sum_{m_1=0}^{2^{j_1}-1}\sum_{\substack{m_2=0 \\ \eqref{system} \text{\, satisfied}}}^{2^{j_2}-1}\mu_{\bsj,\bsm}^2
				  +2^{j_1}(2^{j_2}-2^{n-j_1})2^{-4j_1-4j_2-8}.
 \end{align*}
Note that for a fixed $m_1\in\DD_{j_1}$ the system~\eqref{system} fixes the digits $s_{n-j_1+1},\dots,s_{j_2}$ and thus the digits $s_1,\dots,s_{n-j_1}$ remain free. This means that
there are $2^{n-j_1}$ elements in $\DD_{j_2}$ which satisfy $\eqref{system}$, whereas the remaining $2^{j_2}-2^{n-j_1}$ elements do not. This is where the factor $2^{j_2}-2^{n-j_1}$ in the
last expression comes from. Let us study $$\sum_{m_1=0}^{2^{j_1}-1}\sum_{\substack{m_2=0 \\ \eqref{system} \text{\, satisfied}}}^{2^{j_2}-1}\mu_{\bsj,\bsm}^2.$$
It equals
\begin{align*}
  \sum_{r_2,\dots,r_{j_1}=0}^{1}&\sum_{s_1,\dots,s_{n-j_1}}^{1}\bigg(2^{-n-j_1-j_2-2}(1-|1-s_{n-j_1}\oplus\sigma_{n-j_1}-\dots-2^{j_1-n+1}(s_1\oplus\sigma_1)|) \\
	&\times (1-|1-r_{n-j_2}\oplus \sigma_{j_2+1}-\dots-2^{j_2-n+1}\sigma_n|)-2^{-2j_1-2j_2-4}\bigg)^2  \\
	+\sum_{r_2,\dots,r_{j_1}=0}^{1}&\sum_{s_1,\dots,s_{n-j_1}}^{1}\bigg(2^{-n-j_1-j_2-2}(1-|1-s_{n-j_1}\oplus\sigma_{n-j_1}'-\dots-2^{j_1-n+1}(s_1\oplus\sigma_1')|) \\
	&\times (1-|1-r_{n-j_2}\oplus \sigma_{j_2+1}'-\dots-2^{j_2-n+1}\sigma_n'|)-2^{-2j_1-2j_2-4}\bigg)^2=:S_1+S_2,						
\end{align*}
where we already summed over $r_1$. The sums $S_1$ and $S_2$ can be computed similarly. Note that the summands in $S_1$ do not depend on the digits $r_{n-j_2+1},\dots,r_{j_1}$.
Summation over $r_{n-j_2}$ and $s_{n-j_1}$ leads to
\begin{align*}
  S_1=& 2^{j_1+j_2-n}\sum_{r_2,\dots,r_{n-j_2-1}=0}^{1}\sum_{s_1,\dots,s_{n-j_1-1}=0}^{1}\bigg\{
	    \left(2^{-n-j_1-j_2-2}u(v+2^{j_2-n+1}\sigma_n)-2^{-2j_2-2j_2-4}\right)^2 \\
			&+\left(2^{-n-j_1-j_2-2}u(1-v-2^{j_2-n+1}\sigma_n)-2^{-2j_2-2j_2-4}\right)^2 \\
			&+\left(2^{-n-j_1-j_2-2}(1-u)(v+2^{j_2-n+1}\sigma_n)-2^{-2j_2-2j_2-4}\right)^2 \\
			&+\left(2^{-n-j_1-j_2-2}(1-u)(1-v-2^{j_2-n+1}\sigma_n)-2^{-2j_2-2j_2-4}\right)^2\bigg\},
\end{align*}
where $u=2^{-1}s_{n-j_1-1}\oplus\sigma_{n-j_1-1}+\dots+2^{j_1-n+1}s_1\oplus\sigma_1$ and $v=r_{n-j_2-1}\oplus \sigma_{j_2+2}+\dots+2^{j_2-n+1}\sigma_n$. To compute the sum over the remaining digits,
we replace $u$ by $2^{-n+j_1+1}l_1$ and $v$ by $2^{-n+j_2-2}l_2$ and let $l_1$ run from $0$ to $2^{n-j_1-1}-1$ and $l_2$ run from $0$ to $2^{n-j_2-2}-1$, respectively.
This yields
\begin{align*}
  S_1=& -2^{-3j_1-3j_2-8}+\frac19 2^{-5n-1}+\frac19 2^{-3n-2j_1-2}+\frac19 2^{-3n-2j_2-4}+\frac19 2^{-n-2j_1-2j_2-5}\\
	  &+2^{n-4j_1-4j_2-9}-\sigma_n \left(\frac13 2^{-5n-2}+\frac13 2^{-3n-2j_1-3}\right).
\end{align*}
We obtain a similar result for $S_2$ with the only difference that $\sigma_n$ is replaced by $1-\sigma_n$. Putting all previous results together,
we find
 \begin{align*} \sum_{\bsm\in\DD_{\bsj}}|\mu_{\bsj,\bsm}|^2=&-2^{-3j_1-3j_2-7}+\frac19 2^{-5n-2}+\frac19 2^{-3n-2j_1-3}+\frac19 2^{-3n-2j_2-3}\\ &+\frac19 2^{-n-2j_1-2j_2-4}
	  +2^{n-4j_1-4j_2-8}+2^{j_1}(2^{j_2}-2^{n-j_1})2^{-4j_1-4j_2-8}. \end{align*}
		The rest follows by a straightforward summation of $2^{|\bsj|}\sum_{\bsm\in\DD_{\bsj}}\mu_{\bsj,\bsm}^2$ over all $\bsj\in\mathcal{J}_{11}$.
\end{proof}

\paragraph{Case 13: $\bsj\in\mathcal{J}_{13}:=\{(j_1,j_2)\in \NN_0^2: j_1\geq n \text{\, or \,} j_2\geq  n \}$}

\begin{proposition}
  Let $\bsj\in \mathcal{J}_{13}$ and $\bsm\in \DD_{\bsj}$. Then we have
     $$ \mu_{\bsj,\bsm}=-2^{-2j_1-2j_2-4}. $$
\end{proposition}

\begin{proof}
  No point lies in the interior of $I_{\bsj,\bsm}$ if $j_1\geq n \text{\, or \,} j_2\geq  n$, and hence the result follows directly from~\eqref{art4}.
\end{proof}

Since the Haar coefficients in this case are independent of $\bsm$, the following consequence is easy to verify.
\begin{lemma}
 We have 
 $$\sum_{\bsj\in \mathcal{J}_{13}}2^{|\bsj|}\sum_{\bsm\in\DD_{\bsj}}|\mu_{\bsj,\bsm}|^2=\frac19 2^{-4n-4} (2^{2n+1}-1).$$
\end{lemma}

\section{The Haar coefficients of the symmetrized digital nets}

From the construction $\widetilde{\cP}_{\bsa}(\vecs)=\cP_{\bsa}(\vecs)\cup \cP_{\bsa}(\vecs^{*})$, it is easy to see that for the Haar coefficients $\tilde{\mu}_{\bsj,\bsm}$ of $\Delta(\cdot,\widetilde{\cP}_{\bsa}(\vecs))$
we have $\tilde{\mu}_{\bsj,\bsm}=\mu_{\bsj,\bsm}^{\vecs}+\mu_{\bsj,\bsm}^{\vecs*}$ (compare~\cite[Proof of Lemma 3]{HKP14}). Here, $\mu_{\bsj,\bsm}^{\vecs}$ denote the Haar coefficients of $\Delta(\cdot,\cP_{\bsa}(\vecs))$ and $\mu_{\bsj,\bsm}^{\vecs^*}$ those of $\Delta(\cdot,\cP_{\bsa}(\vecs^*))$. Hence, it is an easy task to derive the Haar coefficients  $\tilde{\mu}_{\bsj,\bsm}$ from our previous results.

\begin{proposition}
  Let $\bsj\in \NN_{-1}^2$ and $\bsm\in\DD_{\bsj}$. Then we have
	\begin{itemize}
	   \item if $\bsj\in \mathcal{J}_1$: $\tilde{\mu}_{\bsj,\bsm}=\frac{1}{2^{n+1}}+\frac{1}{2^{2n+2}}.$
		 \item if $\bsj\in \mathcal{J}_2$: $ \tilde{\mu}_{\bsj,\bsm}=\frac{1}{2^{2n+3}}\left(2-\frac{1}{2^{2j_2-n}}\right)-\frac{1+a_{j_2+1}(2(\sigma_{j_2+1}\oplus \sigma_n) -1)}{2^{2n+2}}. $
		 \item if $\bsj\in \mathcal{J}_3$: $ \tilde{\mu}_{\bsj,\bsm}=-\frac{1}{2^{3n+1}}+\frac{1}{2^{2n+2}}\sum_{k=1}^{n-1}\frac{a_k (1- s_k  \oplus \sigma_k \oplus \sigma_n)}{2^{n-k}}. $
		 \item if $\bsj\in \mathcal{J}_4$ or $\bsj\in \mathcal{J}_7$ : $ \tilde{\mu}_{\bsj,\bsm}=-2^{-2j_i-3}$, with $i=1$ or $i=2$, respectively.
		 \item if $\bsj\in \mathcal{J}_5$: $ \tilde{\mu}_{\bsj,\bsm}=-\frac{1}{2^{n+3}}. $
		 \item if $\bsj\in \mathcal{J}_6$: $ \tilde{\mu}_{\bsj,\bsm}=-\frac{1}{2^{n+2j_1+3}}. $
		 \item if $\bsj\in \mathcal{J}_8$: $ \tilde{\mu}_{\bsj,\bsm}=\frac{1}{2^{2n+2}}(\sigma_{j_2+1}+\sigma_{j_2+1}'-1)(2\sigma_n-1). $
		 \item if $\bsj\in \mathcal{J}_9$: $ \tilde{\mu}_{\bsj,\bsm}=\frac{1}{2^{2n+2}}(2(a_{n-j_1}r_1\oplus\sigma_{j_2+1})-1)(2(a_{j_2+1}r_1\oplus\sigma_{n-j_1})-1). $
		 \item if $\bsj\in \mathcal{J}_{10}$: $\tilde{\mu}_{\bsj,\bsm}=-(-1)^{\sigma_n}2^{-2n-2}\sum_{1}^{2^{2n+2}}\frac{(1-a_k)(2(s_k\oplus \sigma_k)-1)}{2^{n-k}}$.
		 \item if $\bsj\in \mathcal{J}_{11}$: We have
		       \begin{align*}
					   \tilde{\mu}_{\bsj,\bsm}=&2^{-2n-2}\bigg\{\left(1-\left|1-a_{j_2+1}r_1\oplus \sigma_{j_2+1}-u-2^{-j_1}(r_1\oplus \sigma_n)\right|\right)v\\
						                         &+\left(1-\left|1-a_{j_2+1}r_1\oplus \sigma_{j_2+1}\oplus 1-u-2^{-j_1}(r_1\oplus \sigma_n)\right|\right)(1-v)\bigg\} \\
																		&+2^{-2n-2}\bigg\{\left(1-\left|1-a_{j_2+1}r_1\oplus \sigma_{j_2+1}\oplus 1-u'-2^{-j_1}(r_1\oplus \sigma_n\oplus 1)\right|\right)v\\
						                         &+\left(1-\left|1-a_{j_2+1}r_1\oplus \sigma_{j_2+1}-u'-2^{-j_1}(r_1\oplus \sigma_n\oplus 1)\right|\right)(1-v')\bigg\} -2^{-2n-2},
					 \end{align*}
				  where $u=\sum_{k=2}^{j_1} 2^{k-1-j_1}(r_k\oplus a_{n+1-k}r_1\oplus \sigma_{n+1-k})$, $u'=\sum_{k=2}^{j_1} 2^{k-1-j_1}-u$, and where
					$v=\sum_{k=1}^{j_2} 2^{k-1-j_2}(s_k \oplus a_k r_1 \oplus \sigma_k)$ and $v'=\sum_{k=1}^{j_2} 2^{k-1-j_2}-v$.
		 \item if $\bsj\in \mathcal{J}_{12}$: If $j_1+j_2=n$, then we have
		  \begin{align*}
			  \tilde{\mu}_{\bsj,\bsm}=&2^{-n-j_1-j_2-3}\left(1-\left|1-\sum_{k=1}^{n-j_1}\frac{s_k\oplus a_kr_1\oplus \sigma_k}{2^{n-j_1-k}}\right|\right) \\
				   &\times \left(1-\left|1-\sum_{k=2}^{n-j_2}\frac{r_k\oplus a_{n+1-k}r_1\oplus \sigma_{n+1-k}}{2^{n-j_2-k}}-\frac{r_1\oplus\sigma_n}{2^{n-j_2-1}}\right|\right) \\
					&+2^{-n-j_1-j_2-3}\left(1-\left|1-\sum_{k=1}^{n-j_1}\frac{s_k\oplus a_kr_1\oplus \sigma_k\oplus 1}{2^{n-j_1-k}}\right|\right) \\
				   &\times \left(1-\left|1-\sum_{k=2}^{n-j_2}\frac{r_k\oplus a_{n+1-k}r_1\oplus \sigma_{n+1-k}\oplus 1}{2^{n-j_2-k}}-\frac{r_1\oplus\sigma_n\oplus 1}{2^{n-j_2-1}}\right|\right)-2^{-2j_1-2j_2-4}.
			\end{align*}
		If $j_1+j_2\geq n+1$, then
		     \begin{align*} \tilde{\mu}_{\bsj,\bsm}=&2^{-n-j_1-j_2-3}\left(1-\left|1-\sum_{k=1}^{n-j_1}\frac{s_k \oplus a_kr_1\oplus \sigma_k}{2^{n-j_1-k}}\right|\right) \\
		  &\times  \left(1-\left|1-\sum_{k=2}^{n-j_2}\frac{r_k \oplus a_{n+1-k}r_1\oplus \sigma_{n+1-k}}{2^{n-j_2-k}}-\frac{r_1 \oplus \sigma_n}{2^{n-j_2-1}}\right|\right)-2^{-2j_1-2j_2-4} \end{align*}
		if $s_{\nu}\oplus a_{\nu}r_1 \oplus \sigma_{\nu}=r_{n+1-\mu}$,
		 \begin{align*}\tilde{\mu}_{\bsj,\bsm}=&2^{-n-j_1-j_2-3}\left(1-\left|1-\sum_{k=1}^{n-j_1}\frac{s_k \oplus a_kr_1\oplus \sigma_k\oplus 1}{2^{n-j_1-k}}\right|\right) \\
		  &\times  \left(1-\left|1-\sum_{k=2}^{n-j_2}\frac{r_k \oplus a_{n+1-k}r_1\oplus \sigma_{n+1-k}\oplus 1}{2^{n-j_2-k}}-\frac{r_1 \oplus \sigma_n\oplus 1}{2^{n-j_2-1}}\right|\right)-2^{-2j_1-2j_2-4} \end{align*}
		if $s_{\nu}\oplus a_{\nu}r_1 \oplus \sigma_{\nu}\oplus 1=r_{n+1-\mu}$, and $\tilde{\mu}_{\bsj,\bsm}=-2^{-2j_1-2j_2-4}$ otherwise.
		 \item if $\bsj\in \mathcal{J}_{13}$: $ \tilde{\mu}_{\bsj,\bsm}=-2^{-2j_1-2j_2-4}$.
	\end{itemize}
\end{proposition}
Now we have to calculate $\sum_{\bsj\in \mathcal{J}_i}2^{|\bsj|}\sum_{\bsm\in\DD_{\bsj}} |\tilde{\mu}_{\bsj,\bsm}|^2$ for all $i\in\{1,2,\dots,13\}$. In many cases this is easy, and the argumentation in the more
difficult cases is very similar to what we did in the previous section. We therefore state the following results without proofs.
\begin{lemma}
  We consider a symmetrized net $\widetilde{\cP_{\bsa}}$. Let $\tilde{\mu}_{\bsj,\bsm}$ for $\bsj\in \NN_{-1}^2$ and $\bsm\in\DD_{\bsj}$ be the Haar coefficients of the corresponding discrepancy function. Then we have
	\begin{itemize}
	   \item  $\sum_{\bsj\in \mathcal{J}_1}2^{|\bsj|}\sum_{\bsm\in\DD_{\bsj}} |\tilde{\mu}_{\bsj,\bsm}|^2=\left(\frac{1}{2^{n+1}}+\frac{1}{2^{2n+2}}\right)^2.$
		 \item $\sum_{\bsj\in \mathcal{J}_2}2^{|\bsj|}\sum_{\bsm\in\DD_{\bsj}} |\tilde{\mu}_{\bsj,\bsm}|^2=\frac{1}{3\cdot 2^{4n+4}}(2^{2n}-4)-\frac{(-1)^{\sigma_n}}{2^{3n+4}}L+\frac{1}{2^{4n+6}}\sum_{i=1}^{n-1}a_i2^{2i}$.
		 \item $\sum_{\bsj\in \mathcal{J}_3}2^{|\bsj|}\sum_{\bsm\in\DD_{\bsj}} |\tilde{\mu}_{\bsj,\bsm}|^2=\frac{1}{2^{4n+6}}\sum_{i=1}^{n-1}a_i2^{2i}+\frac{1}{2^{4n+4}}.$
		 \item $\sum_{\bsj\in \mathcal{J}_4}2^{|\bsj|}\sum_{\bsm\in\DD_{\bsj}}|\tilde{\mu}_{\bsj,\bsm}|^2=\sum_{\bsj\in \mathcal{J}_7}2^{|\bsj|}\sum_{\bsm\in\DD_{\bsj}}|\tilde{\mu}_{\bsj,\bsm}|^2=\frac{1}{48\cdot 2^{2n}}$.
		\item $\sum_{\bsj\in \mathcal{J}_5}2^{|\bsj|}\sum_{\bsm\in\DD_{\bsj}}|\tilde{\mu}_{\bsj,\bsm}|^2=\frac{1}{2^{2n+6}}.$
		\item $\sum_{\bsj\in \mathcal{J}_6}2^{|\bsj|}\sum_{\bsm\in\DD_{\bsj}}|\tilde{\mu}_{\bsj,\bsm}|^2=\frac{1}{3\cdot 2^{4n+6}}(2^{2n}-4).$
		\item $\sum_{\bsj\in \mathcal{J}_8}2^{|\bsj|}\sum_{\bsm\in\DD_{\bsj}}|\tilde{\mu}_{\bsj,\bsm}|^2=\frac{1}{3\cdot 2^{4n+6}}(2^{2n}-4)-\frac{1}{2^{4n+6}}\sum_{i=1}^{n-1}a_i2^{2i}$.
		\item $\sum_{\bsj\in \mathcal{J}_9}2^{|\bsj|}\sum_{\bsm\in\DD_{\bsj}}|\tilde{\mu}_{\bsj,\bsm}|^2=\frac{1}{9\cdot 2^{4n+6}}(3n\cdot 2^{2n}-7\cdot 2^{2n}+16).$
		\item $\sum_{\bsj\in \mathcal{J}_{10}}2^{|\bsj|}\sum_{\bsm\in\DD_{\bsj}}|\tilde{\mu}_{\bsj,\bsm}|^2=\frac{1}{3\cdot 2^{4n+6}}(2^{2n}-4)-\frac{1}{2^{4n+6}}\sum_{i=1}^{n-1}a_i2^{2i}$.
		\item $\sum_{\bsj\in \mathcal{J}_{11}}2^{|\bsj|}\sum_{\bsm\in\DD_{\bsj}}|\tilde{\mu}_{\bsj,\bsm}|^2=\frac19 2^{-4n-6}(5\cdot 4^n+4-24n)$.
		\item $\sum_{\bsj\in \mathcal{J}_{12}}2^{|\bsj|}\sum_{\bsm\in\DD_{\bsj}}|\tilde{\mu}_{\bsj,\bsm}|^2=\frac{1}{3\cdot 2^{4n+6}}(n(2^{2n}+8)-2(2^{2n}+2))$.
		\item $\sum_{\bsj\in \mathcal{J}_{13}}2^{|\bsj|}\sum_{\bsm\in\DD_{\bsj}}|\tilde{\mu}_{\bsj,\bsm}|^2=\frac{1}{9\cdot 2^{4n+4}}(2^{2n+1}-1)$.
	\end{itemize}
\end{lemma}
We obtain Theorem 2 via $(L_2(\widetilde{\cP}_{\bsa}(\vecs)))^2=\sum_{i=1}^{13}\sum_{\bsj\in \mathcal{J}_i}2^{|\bsj|}\sum_{\bsm\in\DD_{\bsj}} |\tilde{\mu}_{\bsj,\bsm}|^2$.

\section{Why do (symmetrized) digital nets fail to have the optimal order of $L_2$ discrepancy?} \label{why}

In many previous papers (e.g. \cite{daven,lp}) it has been observed that the reason that a point set fails to have the optimal order of $L_2$ discrepancy can often
be found in the zeroth Fourier coefficient of the corresponding discrepancy function (which is the same as the Haar coefficient for $\bsj=(-1,-1)$).
This recurring phenomenon led to the following conjecture by Bilyk~\cite{billat}:

Whenever an $N$-element point set $\cP$ in $[0,1)^2$ satisfies $L_{\infty}(\cP)\lr (\log{N})/N$ (i.e. its star discrepancy is of best possible order in $N$) and $L_{2}(\cP)\gr (\log{N})/N$, then $\cP$ should also satisfy
  $$ \left|\int_{[0,1)^2} \Delta(\bst,\cP) \rd\bst\right| \gr \frac{\log{N}}{N}. $$
  We can deduce from our previous results that it is not true. Consider the point set $\cP_{\bsone}$, where
  $\bsone=(1,\dots,1)\in \ZZ_2^{n-1}$. Then by Proposition 1 we have $\mu_{(-1,-1),(0,0)}=2^{-2n-2}+5\cdot 2^{-n-3}\leq 1/N$, but
  $L_2(\cP_{\bsone})\gr (\log{N})/N$, which follows from Corollary~\ref{coro1}. Note that $L_{\infty}(\cP_{\bsone})\lr (\log{N})/N$, since $\cP_{\bsone}$ is a $(0,n,2)$ net. Hence $\cP_{\bsone}$ is a counterexample to Bilyk's conjecture. More generally, we observe that none of the nets $\cP_{\bsa}(\vecs)$ achieve the optimal order of $L_2$ discrepancy. The reason for this defect is that for all $\bsa$ at least one of the inequalities $\mu_{(-1,-1),(0,0)}\gr (\log{N})/N$ or $\mu_{(0,-1),(0,0)}\gr (\log{N})/N$ holds; hence in some cases the Haar coefficient for $\bsj=(-1,-1)$ is not the one causing trouble. \\
  We would like to point out that one can find an earlier counterexample to the above conjecture in~\cite{lp}. To state the result, we consider the digital $(0,n,2)$-net
  generated by the matrices $C_1=A_1$ as on page~\pageref{matrixa} and the matrix
  $$ C_2=
\begin{pmatrix}
1 & 0 & 0 & \cdots & 0 & 0 & 0 \\
1 & 1 & 0 & \cdots &  0 & 0 & 0 \\
1 & 0 & 1 & \cdots & 0 & 0 & 0 \\
\vdots & \vdots & \vdots & \ddots & \vdots & \vdots & \vdots & \\
1 & 0 & 0 & \cdots &  1 & 0 & 0 \\
1 & 0 & 0 & \cdots &  0 & 1 & 0 \\
1 & 0 & 0 & \cdots &  0 & 0 & 1 \\
\end{pmatrix},
$$ 
which we call $\cP_c$. We denote its shifted version by $\cP_c(\vecs)$.
The following theorem has been proven by Larcher and Pillichshammer in~\cite[Theorem 1]{lp} and shows that not every symmetrized digital net
achieves the optimal order of $L_2$ discrepancy. Their proof is based on a Walsh function analysis of the discrepancy function.
Here we shall give a new proof based on Haar functions.

\begin{theorem}[Larcher and Pillichshammer]
  The $L_2$ discrepancy of the symmetrized point set $\cP_c^{\sym}:=\cP_c \cup \{(x,1-y):(x,y)\in \cP_c\}$ with $N=2^{n+1}$ elements satisfies
  $$ L_2(\cP_c^{\sym})\gr \frac{\log{N}}{N}. $$
	(Note that $\mu_{(-1,-1),(0,0)}(\Delta(\cdot,\cP_c^{\sym}))=2^{-n-2}$ and $L_{\infty}(\cP_c^{\sym})\lr (\log{N})/N$.)
\end{theorem}

\begin{proof}
Instead of $\cP_c^{\sym}$ we investigate the $L_2$ discrepancy of the point set $\widetilde{\cP}_c(\vecs)=\cP_c(\vecs) \cup \cP_c(\vecs^*)$,
as the difference between $L_2(\cP_c^{\sym})$ and $L_2(\widetilde{\cP}_c(\bszero))$ is at most $2^{-n}$. The argumentation of the last statement can be found in~\cite[Lemma 4]{HKP14}. Let $\mu_{\bsj,\bsm}^{\vecs}$ denote the Haar coefficients of $\Delta(\cdot,\cP_c(\vecs))$ and $\tilde{\mu}_{\bsj,\bsm}^{\vecs}$ the Haar coefficients of $\Delta(\cdot,\widetilde{\cP}_c(\vecs))$.
The idea of the proof is as follows: By Parseval's identity we have $L_2(\widetilde{\cP}_c(\bszero))\geq \tilde{\mu}_{(-1,0),(0,0)}^{\bszero}$. We will show
$\tilde{\mu}_{(-1,0),(0,0)}^{\bszero}\gr \frac{\log{N}}{N}$, which yields the result. \\
In order to calculate $\tilde{\mu}_{(-1,0),(0,0)}^{\vecs}$, we first compute $\mu_{(-1,0),(0,0)}^{\vecs}$ for an arbitrary shift.
We write
  $$ \cP_c(\vecs)=\left\{\bigg(\frac{t_n}{2}+\dots+\frac{t_1}{2^n},\frac{t_1\oplus \sigma_1}{2}+\dots+\frac{t_1\oplus t_n \oplus \sigma_n}{2^n}\bigg):t_1,\dots, t_n \in\{0,1\}\right\}. $$
   For a point $\bsz=(z_1,z_2)\in \cP_c$ we have
\begin{align*}
   &\sum_{\bsz\in\cP_c}(1-z_1)(1-|2m_2+1-2^{j_2+1}z_2|) \\ =&\sum_{t_1,\dots,t_n=0}^1 \left(1-\frac{t_n}{2}-\dots-\frac{t_1}{2^n}\right)\left(1-\left|1-t_1\oplus \sigma_1-\frac{t_1\oplus t_2 \oplus \sigma_2}{2}-\dots-\frac{t_1 \oplus t_n \oplus \sigma_n}{2^{n-1}}\right|\right) \\
   =&\sum_{t_2,\dots,t_n=0}^1 \left\{\left(1-u-\frac{\sigma_1}{2^n}\right)v+\left(1-u-\frac{\sigma_1\oplus 1}{2^n}\right)(1-v)\right\} \\
   =&\sum_{t_2,\dots,t_n=0}^1 \left\{-2^{-2n-2}+2^{-n+1}+2^{-2n+1}\sigma_1-2^{-n+1}u+2v-2^{-n}v-2uv\right\},
\end{align*}
where $u=2^{-1}t_n+\dots+2^{-n+1}t_2$ and $v=2^{-1}(t_1\oplus t_2 \oplus \sigma_2)+\dots+2^{-n+1}(t_1\oplus t_n \oplus\sigma_n)$.
We have $\sum_{t_2,\dots,t_n=0}^1 u=\sum_{t_2,\dots,t_n=0}^1 v=2^{n-2}-2^{-1}$; hence it remains to investigate $\sum_{t_2,\dots,t_n=0}^1 uv$. We have
\begin{align*}
  \sum_{t_2,\dots,t_n=0}^1 uv=& \sum_{k=2}^n \sum_{t_2,\dots,t_n=0}^1 \frac{t_k(t_k\oplus \sigma_k \oplus \sigma_1)}{2^{n+1-k}2^{k-1}}+
       \sum_{\substack{k_1,k_2=2 \\ k_1\neq k_2}}^n \frac{t_{k_1}(t_{k_2}\oplus \sigma_{k_2}\oplus \sigma_1)}{2^{n+1-k_1}2^{k_2-1}} \\
       =& \frac{1}{2^n}\sum_{k=2}^n 2^{n-2} \sum_{t_k=0}^1 t_k (t_k \oplus \sigma_k\oplus \sigma_1)\\&+\frac{1}{2^n}\sum_{\substack{k_1,k_2=2 \\ k_1\neq k_2}}^n2^{k_1-k_2}2^{n-3}\sum_{t_{k_1},t_{k_2}=0}^1 t_{k_1}(t_{k_2}\oplus \sigma_{k_2}\oplus \sigma_1) \\
       =&\frac14 \sum_{k=2}^n (1 \oplus \sigma_k\oplus \sigma_1)+\frac18 \sum_{\substack{k_1,k_2=2 \\ k_1\neq k_2}}^n2^{k_1-k_2}.
\end{align*}
Combining our results with~\eqref{art3} yields $$\mu_{(-1,0),(0,0)}=2^{-2n-2}-2^{-n-3}n-2^{-2n-1}\sigma_1+2^{-n-2}\sum_{k=2}^n 2^{n-2} (1 \oplus \sigma_k\oplus \sigma_1).$$
Since $\tilde{\mu}_{(-1,0),(0,0)}^{\vecs}=\frac12(\mu_{(-1,0),(0,0)}^{\vecs}+\mu_{(-1,0),(0,0)}^{\vecs^*})$ we derive
$$\tilde{\mu}_{(-1,0),(0,0)}^{\vecs}=-2^{-n-3}n+2^{-n-2}\sum_{k=2}^n  (1 \oplus \sigma_k\oplus \sigma_1).$$
In particular, for $\vecs=\bszero$ we find $\widetilde{\mu}_{(-1,0),(0,0)}^{\bszero}=2^{-n-3}(n-2)\gr \frac{\log{N}}{N}$ and we are done.
\end{proof}

\section{Further results}

Our method as presented in this paper is not restricted to the class of digital nets $\cP_{\bsa}(\vecs)$. For instance, it is near at hand to also study
the nets $\cP_{\bsc}(\vecs)$ generated by $C_1=A_1$ as on page~\pageref{matrixa} and
$$ C_2=
\begin{pmatrix}
1 & 0 & 0 & \cdots & 0 & 0 & 0 \\
c_2 & 1 & 0 & \cdots &  0 & 0 & 0 \\
c_3 & 0 & 1 & \cdots & 0 & 0 & 0 \\
\vdots & \vdots & \vdots & \ddots & \vdots & \vdots & \vdots & \\
c_{n-2} & 0 & 0 & \cdots &  1 & 0 & 0 \\
c_{n-1} & 0 & 0 & \cdots &  0 & 1 & 0 \\
c_n & 0 & 0 & \cdots &  0 & 0 & 1 \\
\end{pmatrix},
$$ 
where we write $\bsc=(c_2,\dots,c_n)$ and again we apply a digital shift $\vecs=(\sigma_1,\dots,\sigma_n)$ to the second components of the points. We simply write $\cP_{\bsc}$ if we do not apply a shift. Further we put $\widetilde{\cP}_{\bsc}(\vecs):=\cP_{\bsc}(\vecs) \cup \cP_{\bsc}(\vecs^{\ast})$. There are many parallel tracks in the computation of the Haar coefficients of $\Delta(\cdot,\cP_{\bsa}(\vecs))$ and $\Delta(\cdot,\cP_{\bsc}(\vecs))$. We leave it as a (tedious) exercise to the reader to show the following theorem
with the method demonstrated in Section~\ref{haark}.

\begin{theorem} \label{theo3}
    Let $L=\sum_{i=2}^{n}c_i (1-2\sigma_i)$ and $\ell=\sum_{i=1}^n (1-2\sigma_i)$. Then we have
   \begin{align*}
     (2^n\,L_2(\cP_{\bsc}(\vecs)))^2=& \frac{1}{64}\left((\ell-L)^2+L^2+8\ell+2L(2\sigma_1-5)+\frac53 n\right) \\
         &-\frac{1}{2^{n+4}}\left(\ell-4\right)+\frac{3}{8}-\frac19\frac{1}{2^{2n+3}}.
   \end{align*}
\end{theorem}

For the unshifted nets we find a result of the very same form as Corollary~\ref{coro1}.

\begin{corollary} \label{coro2}
   Let $|\bsc|=\sum_{i=2}^{n}c_i$. Then we have
    \begin{align*}
     (2^n\,L_2(\cP_{\bsc}))^2=& \frac{1}{64}\left((n-|\bsc|)^2+|\bsc|^2-10|\bsc|+\frac{29}{3} n\right)+\frac{3}{8}-\frac{n-4}{2^{n+4}}-\frac19\frac{1}{2^{2n+3}}.
   \end{align*}
\end{corollary}
However, there are major differences between the $L_2$ discrepancies of the symmetrized nets $\widetilde{\cP}_{\bsa}(\vecs)$ and $\widetilde{\cP}_{\bsc}(\vecs)$,
as the following theorem demonstrates. Since an exact computation of $\sum_{\bsj\in \mathcal{J}_i}2^{|\bsj|}\sum_{\bsm\in\DD_{\bsj}} |\tilde{\mu}_{\bsj,\bsm}|^2$ for $i\in\{11,12\}$
is very complicated, we avoid an exact calculation of the $L_2$ discrepancy. At least we can show that $$\sum_{\bsj\in \NN_{-1}^2 \setminus \{(-1,0)\}}2^{|\bsj|}\sum_{\bsm\in\DD_{\bsj}} |\tilde{\mu}_{\bsj,\bsm}|^2\lr n/2^{2n}$$ and $2^{|\bsj|}\sum_{\bsm\in\DD_{\bsj}} |\tilde{\mu}_{\bsj,\bsm}|^2=2^{-2n-6} \left(L^2-2(1-2\sigma_1)L+1\right)$ for $\bsj=(-1,0)$. Therefore the following result is a consequence of Parseval's identity.

\begin{theorem}
  Let $L$ be as in Theorem~\ref{theo3}. Then we have
 $L_2(\widetilde{\cP}_{\bsc}(\vecs))\lr \sqrt{\log{N}}$ if and only if $|L|\lr \sqrt{n}$. For the unshifted symmetrized nets
we have $L_2(\widetilde{\cP}_{\bsc})\lr \sqrt{\log{N}}$ if and only if $|\bsc|\lr \sqrt{n}$
\end{theorem}

\section{Results on the $L_p$ discrepancy}

The calculation of the Haar coefficients of the discrepancy allows us to study not only the $L_2$ discrepancy of point sets, but also the
$L_p$ discrepancy for all $p\in (1,\infty)$. The key tool for that purpose is the Littlewood-Paley inequality for Haar functions. It states that
for every function $f$ in $L_p([0,1)^2)$ with $p\in(1,\infty)$ we have the norm equivalence $\|f\|_{L_p([0,1)^2)}\asymp \|S(f)\|_{L_p([0,1)^2)}$, where 
\begin{equation*} \label{squfct} S(f):=\left(\sum_{\bsj\in\NN_{-1}^s,\bsm\in\DD_{\bsj}}2^{2|\bsj|} |\mu_{\bsj,\bsm}|^2 \bsone_{I_{\bsj,\bsm}}\right)^{\frac12}. \end{equation*}		
The Littlewood-Paley inequality enables us to give sufficient and necessary conditions for the points sets we study in this paper to achieve the optimal order of $L_p$ discrepancy. It is not necessary to work with the exact Haar coefficients to show these conditions. The following upper bounds on the Haar coefficients of $\Delta(\cdot,\cP_{\bsa}(\vecs))$ can be derived immediately from the propositions in Section~\ref{haark}.

\begin{corollary}
Let $\mu_{\bsj,\bsm}$ be the Haar coefficients of $\Delta(\cdot,\cP_{\bsa}(\vecs))$.
Let $\bsj=(j_1,j_2)\in \NN_{0}^2$. Then 
\bigskip
\begin{itemize}
  \item[(i)] if $j_1=0$ and $0\leq j_2 \leq n-2$ then $|\mu_{\bsj,\bsm}| \lr 2^{-n-j_2}$
  \item[(ii)] if $j_1+j_2<n-1$ and $j_1,j_2\ge 0$ then $|\mu_{\bsj,\bsm}| = 2^{-2n-2}$.
  \item[(iii)] if $j_1+j_2\ge n-1$ and $0\le j_1,j_2\le n$ then $|\mu_{\bsj,\bsm}| \lr 2^{-n-j_1-j_2}$ and
     $|\mu_{\bsj,\bsm}| = 2^{-2j_1-2j_2-4}$ for all but at most $2^n$ coefficients $\mu_{\bsj,\bsm}$ with $\bsm\in {\mathbb D}_{\bsj}$ (the latter appears if there is no point of $\cP_{\bsa}(\vecs)$ in the interior of $I_{\bsj,\bsm}$).
  \item[(iv)] if $j_1 \ge n$ or $j_2 \ge n$ then $|\mu_{\bsj,\bsm}| = 2^{-2j_1-2j_2-4}$.
 \end{itemize} 
 
 Now let $\bsj=(-1,j_2)$ with $j_2\in \NN_0$. Then
 \begin{itemize}
  \item[(v)] if $j_2<n$ then $|\mu_{\bsj,\bsm}| \lr 2^{-n-j_2}$.
  \item[(vi)] if $j_2\ge n$  then $|\mu_{\bsj,\bsm}| = 2^{-2j_2-3}$.
 \end{itemize}

 Next let $\bsj=(j_1,-1)$ with $j_1\in \NN_0$. Then
 \begin{itemize}
\item[(vii)] if $j_1=0$ then $|\mu_{\bsj,\bsm}| =\frac{1}{2^{2n+2}}-\frac{1}{2^{n+3}}+\frac{1}{2^{n+3}}L-\frac{1}{2^{2n+1}}\sigma_n$.
  \item[(viii)] if $1\leq j_1<n$ then $|\mu_{\bsj,\bsm}| \lr 2^{-n-j_1}$.
  \item[(ix)] if $j_1\ge n$  then $|\mu_{\bsj,\bsm}| = 2^{-2j_1-3}$.
 \end{itemize}

 Finally, if $\bsj=(-1,-1)$ then
  \begin{itemize}
   \item[(vi)] $\mu_{\bsj,\bsm}=\frac{1}{2^{n+1}}+\frac{1}{2^{2n+2}}+\frac{1}{2^{n+3}}(\ell-L).$
  \end{itemize} 
\end{corollary}
We insert these bounds into the Littlewood-Paley inequality to show the following result.
The procedure of the proof is basically the same as in~\cite{HKP14}, where the result has been shown
for the Hammersley point set. Of course we can do the same for the class $\cP_{\bsc}(\vecs)$ of shifted nets.

\begin{theorem}
   Let $\ell$ and $L$ be as in Theorem~\ref{theo1} and $p\in (1,\infty)$. Then we have
	  $$ L_p(\cP_{\bsa}(\vecs))\lr_p \frac{\sqrt{\log{N}}}{N} $$
		if and only if $|\ell-L|\lr_p \sqrt{n}$ and $|L|\lr_p\sqrt{n}$ hold simultaneously. 
		An analogous result holds for the class of point sets $\cP_{\bsc}(\vecs)$
\end{theorem}
For the symmetrized nets we find the following conditions which assure the optimal order of $L_p$ discrepancy.
\begin{theorem}
  Let $p\in(1,\infty)$. We have 
	$$ L_p(\widetilde{\cP}_{\bsa}(\vecs))\lr \frac{\sqrt{\log{N}}}{N} $$
	for all $\bsa\in\ZZ_2^{n-1}$ and all $\vecs\in\ZZ_2^n$.
	We further have
	$$ L_p(\widetilde{\cP}_{\bsc}(\vecs))\lr \frac{\sqrt{\log{N}}}{N} $$
	if and only if $|L|\lr\sqrt{n}$.
\end{theorem}

\section{Outlook}

It would be increasingly difficult to obtain exact formulas for the $L_2$ discrepancy of more complicated digital nets. However,
we could ask for conditions on the matrices $C_1$ and $C_2$ such that the $L_2$ discrepancy of the digital net generated by these is of optimal order~\eqref{roth}.
Let us, for instance, consider the digital $(0,n,2)$-net $\cP_{tri}$ generated by $C_1=A_1$ as on page~\pageref{matrixa} and
$$ C_2=
\begin{pmatrix}
1 & a_{1,2} & a_{1,3} & \cdots & a_{1,n-2} & a_{1,n-1} & a_{1,n} \\
0 & 1 & a_{2,3} & \cdots &  a_{2,n-2} & a_{2,n-1} & a_{2,n} \\
0 & 0 & 1 & \cdots & a_{3,n-2} & a_{3,n-1} & a_{3,n} \\
\vdots & \vdots & \vdots & \ddots & \vdots & \vdots & \vdots & \\
0 & 0 & 0 & \cdots &  1 & a_{n-2,n-1} & a_{n-2,n} \\
0 & 0 & 0 & \cdots &  0 & 1 & a_{n-1,n} \\
0 & 0 & 0 & \cdots &  0 & 0 & 1 \\
\end{pmatrix}.
$$ 
For the digital nets $\cP_{\bsa}$ we observed that either $\mu_{(-1,-1),(0,0)}(\Delta(\cdot,\cP_{\bsa}))\gr (\log{N})/N$ or $\mu_{(0,-1),(0,0)}(\Delta(\cdot,\cP_{\bsa}))\gr (\log{N})/N$ holds. If we could show that a similar result holds for $\Delta(\cdot,\cP_{tri})$, then we would know that all the nets $\cP_{tri}$ fail to achieve the optimal order of $L_2$ discrepancy as well.
However, this is not the case in general. We define several parameters to demonstrate this claim: For $\mu\in\{1,\dots,n\}$ put $l_{\mu}(\mu):=1$ and for $k\in\{1,\dots,\mu-1\}$ we put
  $$ l_{\mu}(k):=\begin{cases}
	          0, & \text{if\, } \exists i\in \{k+1,\dots,\mu\}: a_{k,i}=1, \\
						1, & \text{if\, } \forall i\in \{k+1,\dots,\mu\}: a_{k,i}=0.
	        \end{cases}$$
					Then a direct computation similar to the proofs of Propositions~\ref{prop1} and~\ref{prop5} yields
					$$  \mu_{(-1,-1),(0,0)}(\Delta(\bst,\cP_{tri}))=\frac{1}{2^{n+3}}\sum_{k=1}^n l_n(k)+\frac{1}{2^{n+1}}+\frac{1}{2^{2n+2}}  $$
					and
					$$  \mu_{(0,-1),(0,0)}(\Delta(\bst,\cP_{tri}))=\frac{1}{2^{n+3}}\left(\sum_{k=1}^{n-1}l_{n-1}(k)-\sum_{k=1}^n l_n(k)\right)+\frac{1}{2^{2n+2}}. $$
Hence we have both $\mu_{(-1,-1),(0,0)}(\Delta(\cdot,\cP_{tri}))\lr 1/N$ and $\mu_{(0,-1),(0,0)}(\Delta(\cdot,\cP_{tri}))\lr 1/N$ if we choose for $C_2$ the
particular matrices
$$ 
\begin{pmatrix}
1 & 0 & 0 & \cdots & 0 & 1 & 1 \\
0 & 1 & 0 & \cdots &  0 & 1 & 1 \\
0 & 0 & 1 & \cdots & 0 & 1 & 1 \\
\vdots & \vdots & \vdots & \ddots & \vdots & \vdots & \vdots & \\
0 & 0 & 0 & \cdots &  1 & 1 & 1 \\
0 & 0 & 0 & \cdots &  0 & 1 & 1 \\
0 & 0 & 0 & \cdots &  0 & 0 & 1 \\
\end{pmatrix} \text{\, or \,}
\begin{pmatrix}
1 & 1 & 0 & \cdots & 0 & 0 & 0 \\
0 & 1 & 1 & \cdots &  0 & 0 &0 \\
0 & 0 & 1 & \cdots & 0 & 0 & 0 \\
\vdots & \vdots & \vdots & \ddots & \vdots & \vdots & \vdots & \\
0 & 0 & 0 & \cdots &  1 & 1 & 0 \\
0 & 0 & 0 & \cdots &  0 & 1 & 1 \\
0 & 0 & 0 & \cdots &  0 & 0 & 1 \\
\end{pmatrix}.
$$ 
We assume that we achieve the lowest possible $L_2$ discrepancy for the net $\cP_{tri}$ if we fill the whole upper right triangle of $C_2$ with ones. We will investigate in future research whether the corresponding digital net achieves the optimal order of $L_2$ discrepancy without shifting or symmetrization and we hope to be able to determine precisely which conditions on the matrix $C_2$ can achieve that.

\noindent{\bf Author's Address:}

\noindent Ralph Kritzinger, Institut f\"{u}r Finanzmathematik und angewandte Zahlentheorie, Johannes Kepler Universit\"{a}t Linz, Altenbergerstra{\ss}e 69, A-4040 Linz, Austria. Email: ralph.kritzinger(at)jku.at
  
\end{document}